\documentclass[12pt,a4paper]{amsart}

\usepackage{amssymb}
\usepackage[utf8]{inputenc}
\usepackage{enumerate,ushort,bbm,titletoc,a4}
\usepackage[pagebackref,colorlinks,linkcolor=red,citecolor=blue,urlcolor=blue,hypertexnames=true]{hyperref}

\DeclareMathOperator{\End}{End}
\DeclareMathOperator{\Hom}{Hom}
\DeclareMathOperator{\sign}{sign}

\DeclareMathOperator{\supp}{supp}
\DeclareMathOperator{\dom}{dom}
\DeclareMathOperator{\edom}{edom}
\DeclareMathOperator{\edomst}{edom^{st}}
\DeclareMathOperator{\lin}{Span}
\newcommand{\der}{\mathrm d}
\newcommand{\N}{\mathbb{N}}
\newcommand{\F}{\mathbb{F}}

\newcommand{\R}{\mathbb{R}}
\newcommand{\C}{\mathbb{C}}
\newcommand{\rr}{\mathbbm{r}}
\newcommand{\rs}{\mathbbm{s}}
\newcommand{\cA}{\mathcal{A}}
\newcommand{\cB}{\mathcal{B}}
\newcommand{\cH}{\mathcal{H}}
\newcommand{\cM}{\mathcal{M}}
\newcommand{\cN}{\mathcal{N}}
\newcommand{\cO}{\mathcal{O}}
\newcommand{\cR}{\mathcal{R}}
\newcommand{\cU}{\mathcal{U}}
\newcommand{\cV}{\mathcal{V}}
\newcommand{\cW}{\mathcal{W}}

\newcommand{\bb}{\mathbf{b}}
\newcommand{\cc}{\mathbf{c}}
\newcommand{\uu}{\mathbf{u}}
\newcommand{\vv}{\mathbf{v}}
\newcommand{\fW}{\mathfrak{W}}
\newcommand{\fX}{\mathfrak{\ulx}}
\newcommand{\fx}{\mathfrak{x}}
\newcommand{\ud}{\operatorname{UD}_{m;g}}

\newcommand{\ulT}{\ushort T}
\newcommand{\ula}{\ushort a}
\newcommand{\ulb}{\ushort b}
\newcommand{\ulp}{\ushort p}
\newcommand{\ulq}{\ushort q}
\newcommand{\ulx}{\ushort x}
\newcommand{\ult}{\ushort t}
\newcommand{\uly}{\ushort y}
\newcommand{\ulz}{\ushort z}
\newcommand{\ulfx}{\ushort \fx}
\newcommand{\ulxi}{\ushort \xi}

\newcommand*{\edomstn}[1]{\edom^{\operatorname{st}}_{#1}}

\newcommand{\Langle}{\langle}
\newcommand{\Rangle}{\rangle}

\newcommand{\gp}{\Langle \ulx \Rangle}
\newcommand{\gpfx}{\Langle \ulfx \Rangle}
\newcommand{\agp}{\cA\gp}
\newcommand{\gs}{\Langle\!\Langle \ulx \Rangle\!\Rangle}
\newcommand{\ags}{\cA\gs}
\newcommand{\eqind}{\overset{\underset{\mathrm{ind}}{}}{=}}

\makeatletter
\def\moverlay{\mathpalette\mov@rlay}
\def\mov@rlay#1#2{\leavevmode\vtop{
    \baselineskip\z@skip \lineskiplimit-\maxdimen
    \ialign{\hfil$#1##$\hfil\cr#2\crcr}}}
\makeatother

\newcommand{\plangle}{\moverlay{(\cr<}}
\newcommand{\prangle}{\moverlay{)\cr>}}

\newcommand{\pz}{\Langle \ulz \Rangle}
\newcommand{\rf}{\F\plangle \ulz \prangle}
\newcommand{\re}{\cR_\F(\ulz)}

\newtheorem{thm}{Theorem}[section]
\newtheorem{lem}[thm]{Lemma}
\newtheorem{cor}[thm]{Corollary}
\newtheorem{prop}[thm]{Proposition}
\theoremstyle{definition}
\newtheorem{defn}[thm]{Definition}
\newtheorem{exa}[thm]{Example}
\theoremstyle{remark}
\newtheorem{rem}[thm]{Remark}

\numberwithin{equation}{section}

\newcounter{Inc}

\begin{document}

\setcounter{tocdepth}{3}
\contentsmargin{2.55em} 
\dottedcontents{section}[3.8em]{}{2.3em}{.4pc} 
\dottedcontents{subsection}[6.1em]{}{3.2em}{.4pc}
\dottedcontents{subsubsection}[8.4em]{}{4.1em}{.4pc}

\makeatletter
\newcommand{\mycontentsbox}{%
{\centerline{NOT FOR PUBLICATION}
\addtolength{\parskip}{-2.3pt}
\tableofcontents}}
\def\enddoc@text{\ifx\@empty\@translators \else\@settranslators\fi
\ifx\@empty\addresses \else\@setaddresses\fi
\newpage\mycontentsbox\newpage\printindex}
\makeatother

\setcounter{page}{1}

\title[Matrix coefficient realization theory]{Matrix coefficient realization theory of noncommutative rational functions}

\author[J. Vol\v{c}i\v{c}]{Jurij Vol\v{c}i\v{c}${}^\star$}
\address{Jurij Vol\v{c}i\v{c}, The University of Auckland, Department of Mathematics}
    \email{jurij.volcic@auckland.ac.nz}
\thanks{${}^\star$Research supported by The University of Auckland Doctoral Scholarship.}

\date{\today}

\begin{abstract}
Noncommutative rational functions, i.e., elements of the universal skew field of fractions of a free algebra, can be defined through evaluations of noncommutative rational expressions on tuples of matrices. This interpretation extends their traditionally important role in the theory of division rings and gives rise to their applications in other areas, from free real algebraic geometry to systems and control theory. If a noncommutative rational function is regular at the origin, it can be described by a linear object, called a \emph{realization}. In this article we present an extension of the realization theory that is applicable to \emph{arbitrary} noncommutative rationas functions and is well-adapted for studying matrix evaluations.

Of special interest are the minimal realizations, which compensate the absence of a canonical form for noncommutative rational functions. The non-minimality of a realization is assessed by obstruction modules associated with it; they enable us to devise an efficient method for obtaining minimal realizations. With them we describe the stable extended domain of a noncommutative rational function and define a numerical invariant that measures its complexity. Using these results we determine concrete size bounds for rational identity testing, construct minimal symmetric realizations and prove an effective local-global principle for linear dependence of noncommutative rational functions.
\end{abstract}

\maketitle



\section{Introduction}

As the universal skew field of fractions of a free algebra, noncommutative rational functions naturally play a prominent role in the theory of division rings and in noncommutative algebra in general \cite{Ami, Co, Le, Li}. On the other hand, they also arise in other areas, such as free real algebraic geometry \cite{HMV, OHMP, HKM}, algebraic combinatorics \cite{GGRW, GKLLRT}, systems theory \cite{BGM, BGM1, AM}, automata theory \cite{Sch, Fl, BR} and free analysis \cite{KVV2, KVV3}.

One of the main difficulties about working with noncommutative rational functions is that they lack a canonical form. For noncommutative rational functions analytic at 0 this can be resolved by introducing \emph{linear representations} or \emph{realizations}, as they are called in automata theory and control theory, respectively: if $\rr$ is a noncommutative rational function in $g$ arguments with coefficients in a field $\F$ and defined at $(0,\dots,0)$, then there exist $n\in\N$, $\cc\in \F^{1\times n}$, $\bb\in \F^{n\times 1}$, and $A_j\in\F^{n\times n}$ for 
$1\le j\le g$, such that
$$\rr(z_1,\dots,z_g)=\cc\left(I_n-\sum_{j=1}^g A_jz_j\right)^{-1}\bb.$$
In general, $\rr$ admits various realizations; however, those with minimal $n$ are similar up to conjugation (see Definition \ref{d:recognizable}) and thus unique in some sense. Throughout the paper we address these results as the \emph{classical} representation or realization theory and refer to \cite{BGM, BR} as the main sources. Applications of such realizations also appear outside control and automata theory, for example in free probability \cite{And, BMS}. Of course, this approach leaves out noncommutative rational functions that are not defined at any scalar point, e.g. $(z_1z_2-z_2z_1)^{-1}$. 

One way of adapting to the general case is to consider realizations of noncommutative rational functions over an infinite-dimensional division ring as in \cite[Section 7.6]{Co1} or \cite{CR}. The latter paper considers generalized series over an infinite-dimensional division ring and corresponding realization theory. However, in the aforementioned setting of systems theory, free (real) algebraic geometry and free analysis, noncommutative rational functions are applied to tuples of matrices or operators on finite-dimensional spaces. Hence the need for a representation of noncommutative rational functions that is adapted to the matrix setting.

The aim of this paper is to develop a comprehensive realization theory that is applicable to arbitrary noncommutative rational functions and is based on matrix-valued evaluations. The basic idea is to expand a noncommutative rational function in a generalized power series about some matrix point in its domain. Analogously to the classical setting, we consider linear representations of rational generalized series; we reserve the term ``realization'' for a representation corresponding to an expansion as above. If a noncommutative rational function $\rr$ is defined at $\ulp=(p_1,\dots,p_g)\in M_m(\F)^g$, then there exist $n\in\N$, 
$\cc\in M_m(\F)^{1\times n}$, $\bb\in M_m(\F)^{n\times 1}$, and 
$C_{jk},B_{jk}\in M_m(\F)^{n\times n}$ (finite number of them) such that
\begin{equation}\label{e:sylvester}
\rr(z_1,\dots,z_g)=\cc\left(I_n-\sum_{j=1}^g\sum_k C_{jk}(z_j-p_j)B_{jk}\right)^{-1}\bb
\end{equation}
holds under the appropriate interpretation (see Subsection \ref{subs:def} for more precise statement). Such a description is called a \emph{Sylvester realization} of $\rr$ about $\ulp$ of dimension $n$. Of particular interest are the Sylvester realizations of $\rr$ with the smallest possible dimension.

\subsection{Main results and reader's guide}

In Section \ref{sec2} we introduce noncommutative rational functions as equivalence classes of noncommutative rational expressions with respect to their evaluations on tuples of matrices.

In Section \ref{sec3} we define generalized series and linear representations. By Theorem \ref{t:equiv}, rational generalized series are precisely those that admit a linear representation; a concrete construction is given in Subsection \ref{subs:arith}. Evaluations of rational generalized series over a matrix ring are discussed in Subsection \ref{subs:eval}.

Section \ref{sec4} is concerned with the properties of reduced, minimal, and totally reduced representations. Proposition \ref{p:red} gives us an algorithm to produce reduced representations, while Theorem \ref{t:redmin} implies that they are not far away from being minimal. Subsection \ref{subs:hankel} offers a module theoretic replacement for the classical Hankel matrix.

Our main results regarding noncommutative rational functions are given in Section \ref{sec5}. After defining the realization of a noncommutative rational function about a matrix point in its domain, we prove in Theorem \ref{t:totred} that we obtain totally reduced realizations about ``almost every'' point using the reduction algorithm from Section \ref{sec4}. By Theorem \ref{t:indepmin}, the dimensions of minimal realizations are independent of the choice of the expansion point. Thus we can define the \emph{Sylvester degree} of a noncommutative rational function $\rr$ as the dimension of its minimal realization, which in some sense measures $\rr$'s complexity. Furthermore, Corollary \ref{c:dom} describes the stable extended domain of a noncommutative rational function using its totally reduced realization.

Applications of the derived theory are given in Section \ref{sec6}. In Subsection \ref{subs:RI} we address rational identities and provide explicit size bounds for the rational identity testing problem that can be stated without the notion of realization. Next we obtain an algebraic proof of the local-global principle for linear dependence of noncommutative rational functions \cite{CHSY}. In contrast with previous proofs, we obtain concrete size bounds for testing linear dependence; see Theorem \ref{t:lg}. Lastly, Subsection \ref{subs:sym} is devoted to symmetric realizations whose existence and properties are gathered in Theorem \ref{t:sym}.

\subsection{Differences with the classical theory}

The cornerstone of this paper are recognizable generalized series and their linear representations, which can also be of independent interest. It is thus understandable that there is some resemblance with the classical theory. However, there are several obstacles arising from working with matrix rings instead of fields which do not appear in the classical theory \cite{BR,BGM}.

The notions and results from Section \ref{sec3} on noncommutative generalized series over an arbitrary algebra are a natural generalization of the corresponding results for noncommutative power series over a field as presented e.g. in \cite[Chapter 1]{BR}. The biggest change required is in the definitions of shift operators and stable modules since the variables and coefficients do not commute in our setting.

The process of minimization, which is treated in Section \ref{sec4}, demands much more care when dealing with the coefficients belonging to a matrix ring. We have to operate with modules, not vector spaces as in the classical case. While the latter always have bases, this is not true for modules over matrix rings, so instead we have to work with minimal generating sets and maximal linearly independent sets. This has serious consequences. For example, reduced representations (also called \emph{controllable and observable} realizations \cite[Sections 5 and 6]{BGM} in control theory) and minimal representations coincide in the classical theory, whence we have a method of producing minimal representations. In our setting, controllability and observability are replaced by conditions on \emph{obstruction modules}: a realization is reduced if its obstruction modules are torsion modules, and is totally reduced if its obstruction modules are trivial; see Definition \ref{d:red}. Totally reduced representations are minimal, and the latter are reduced. However, in general these families of linear representations are distinct. Furthermore, while we can use an efficient algorithm to produce reduced representation, this in general does not hold for minimal ones. On the other hand, all totally reduced representations are similar and this fails for arbitrary minimal representations. We refer to Subsection \ref{subs:rmtr} for precise statements and examples.

\subsection*{Acknowledgment}

The author would like to thank his supervisor Igor Klep for many long discussions and his fruitful suggestions.



\section{Preliminaries}\label{sec2}

We begin by establishing the general notation and terminology regarding noncommutative rational functions. We define them through the evaluations of noncommutative rational expressions on matrices, 
following \cite[Section 2]{KVV2} (also cf. \cite[Appendix A]{HMV}), where a more detailed analysis can be found. For other classical approaches, see \cite{Row,Co}; alternative constructions of the skew field of noncommutative rational functions employ a free group \cite{Le} or a free Lie algebra \cite{Li}.

Throughout the paper let $\F$ be a field of characteristic 0. This assumption is superfluous to some extent; the reader can observe that the results of Section \ref{sec4} hold for any field $\F$, and that most of the results of Section \ref{sec4} hold even if $\F$ is only a unital commutative ring.

A finite set $\ulz=\{z_1,\dots, z_g\}$ is called an alphabet with letters $z_j$. Let $\pz$ be the free monoid generated by $\ulz$. An element $w\in\pz$ is called a \emph{word} over $\ulz$ and $|w|\in \N\cup\{0\}$ denotes its length. Let $\F\pz$ be the free associative $\F$-algebra generated by $\ulz$; its elements are \emph{(nc) polynomials}. Later on, we also consider alphabets with letters $x_j$ or $y_j$. A syntactically valid combination of nc polynomials with coefficients in $\F$, arithmetic operations $+,\ \cdot,\ {}^{-1}$ and parentheses $(,)$ is called a 
\emph{(nc) rational expression}. The set of all rational expressions is 
denoted $\re$. For example, 
$(1+z_3^{-1}z_2)^{-1}+1$, $z_1+(-z_1)$ and $0^{-1}$ are elements of $\re$.

Let $m\in\N$. Every polynomial $p\in\F\pz$ can be naturally evaluated at a point $\ula=(a_1,\dots,a_g)\in M_m(\F)^g$ by replacing $z_j$ with $a_j$ and $1$ with $I_m$; the result is $p(\ula)\in M_m(\F)$. Moreover, we can naturally extend the evaluations of polynomials to the evaluations of rational expressions. Given $r\in\re$, then $r(\ula)$ is defined in the obvious way if all inverses appearing in $r$ exist at $\ula$. The set of all such $\ula\in M_m(\F)^g$ is denoted 
$\dom_m r$ and called the \emph{domain} of $r$ over $M_m(\F)$. Observe that $\dom_m r$ is an open set in the Zariski topology on $M_m(\F)^g$. If $\dom_m r\neq\emptyset$ for some $m\in\N$, then $\dom_{km} r\neq\emptyset$ for all $k\in\N$.

To acquire the notion of a generic evaluation, we require few concepts from the theory of polynomial identities for matrix rings. Let
$$\F[\ult]=\F\left[t_{ij}^{(k)}\colon 1\le i,j\le m,\, 1\le k\le g\right]$$
be the ring of polynomials in $gm^2$ commutative indeterminates and let $\F(\ult)$ be its field of fractions. The distinguished matrices
$$T_k=\left(t_{ij}^{(k)}\right)_{ij}\in M_{m}\left(\F[\ult]\right)$$
are called the \emph{generic $m\times m$ matrices} and the unital $k$-subalgebra in $M_n(\F[\ult])$ generated by $T_1,\dots, T_g$ is called the \emph{ring of generic matrices of size $m$} \cite{Pro, Fo}. Its central closure in $M_n(\F(\ult))$ is a division algebra, denoted $\ud$ and called the \emph{generic division algebra of degree $m$}; we refer to \cite[Section 3.2]{Row} or \cite[Chapter 14]{Sa} for a good exposition.

Let $r$ be a nc rational expression; as before, we can attempt to evaluate it on the tuple $\ulT=(T_1,\dots,T_g)$. If $r(\ulT)\in M_{m}\left(\F(\ult)\right)$ exists, then obviously $r(\ulT)\in \ud$; we say that this matrix is the \emph{generic evaluation of $r$ of size $m$}. The intersection of the domains of the entries in $r(\ulT)$ is the \emph{extended domain} of $r$ over $M_m(\F)$ and is denoted $\edom_m r$. Let
$$\dom r=\bigcup_{m\in\N}\dom_m r,\quad \edom r=\bigcup_{m\in\N}\edom_m r.$$
While $\dom r$ is closed under direct sums, that is,
$$\ula',\ula''\in\dom r \ \Rightarrow \ula'\oplus\ula''=\begin{pmatrix}\ula'& 0 \\ 0 & \ula''\end{pmatrix} \in\dom r,$$
this does not hold for $\edom r$. Hence it is more convenient to work with the \emph{stable extended domain} of $r$
$$\edomst r=\bigcup_{m\in\N}\edomstn{m}r,$$
where
$$\edomstn{m}r=\left\{\ula\in \edom_mr\colon
\overbrace{\ula\oplus\dots\oplus\ula}^k\in\edom_{km}r \text{ for all } k\in\N\right\}.$$
By \cite[Proposition 3.3]{Vol}, every stable domain is closed under direct sums. Clearly we have $\dom r\subseteq \edomst r\subseteq \edom r$ for every $r\in\re$.

A rational expression is \emph{degenerate} if $\dom r = \emptyset$ and \emph{nondegenerate} otherwise. On the set of nondegenerate rational expressions we define a relation $r_1\sim r_2$ if and only if 
$r_1(\ula)=r_2(\ula)$ for all $\ula\in\dom r_1\cap \dom r_2$. Since $\F$ is infinite, previous remarks imply that $\sim$ is an equivalence relation. If $r$ is a nondegenerate rational expression and $r\sim 0$, then we say that $r$ is a \emph{rational identity}. Observe that since $\ud$ is a division ring, $r\in \re$ is degenerate or a rational identity if and only if $r^{-1}$ is degenerate.

Finally, we define \emph{(nc) rational functions} as the equivalence classes of nondegenerate rational expressions. By \cite[Proposition 2.1]{KVV2}, they form a division ring, denoted $\rf$. The latter plays an important role in the theory of division rings since it is the universal skew field of the free algebra $\F\pz$ by \cite[Proposition 2.2]{KVV2}; we refer to \cite[Chapter 4]{Co}) for details about noncommutative localizations. For every $\rr\in \rf$ let $\dom_m \rr$ be the union of $\dom_m r$ over all representatives $r\in\re$ of $\rr$. In a similar fashion we define $\edomstn{m} \rr$, $\dom \rr$ and $\edomst \rr$.



\section{Linear representations}\label{sec3}

In this section we introduce generalized series and linear representations. The main result is Theorem \ref{t:equiv}, which shows that the classes of rational and recognizable series coincide. In Subsection \ref{subs:eval} we discuss evaluations of rational series.


\subsection{Generalized series}

We start with basic and quite general concepts that are required throughout the paper. Let $\cA$ be a central unital $\F$-algebra and let $\ulx=\{x_1,\dots, x_g\}$ be a finite alphabet, i.e., a set of freely noncommuting letters.

\begin{defn}
The free product of unital $\F$-algebras $\agp = \cA *_{\F} \F\gp$ is called \emph{the algebra of generalized polynomials} over $\cA$. \emph{The algebra of generalized (power) series} over $\cA$, denoted by $\ags$, is defined as the completion of $\agp$ with respect to the $(\ulx)$-adic topology. Comparing with the notation in \cite{Co, Co1} we have $\agp=\cA_{\F}\gp$ and $\ags=\cA_{\F}\gs$.
\end{defn}

If $S\in \ags$ is written as
\begin{equation}\label{e:exp}
S=\sum_{w\in \gp} \sum_{i=1}^{n_w}a_{w,i}^{(0)}w_1a_{w,i}^{(1)}w_2\cdots w_{|w|}a_{w,i}^{(|w|)},
\end{equation}
then let
$$[S,w]=\sum_i a_{w,i}^{(0)}w_1a_{w,i}^{(1)}w_2\cdots w_{|w|}a_{w,i}^{(|w|)}.$$
This is a well-defined element of $\agp$ even though the expansion \eqref{e:exp} is not uniquely determined. The \emph{support} of $S$ is $\supp S=\{w: [S,w]\neq0\}$. A series is a polynomial if and only if its support is finite. It is easy to check that a series $S$ is invertible in $\ags$ if and only if the constant term 
$[S,1]$ is invertible in $\cA$. In particular, if $[T,1]=0$ for $T\in\ags$, then
$$(1-T)^{-1}=\sum_{k\ge0}T^k.$$

\begin{defn}\label{d:rational}
\emph{Rational operations} in $\ags$ are the sum, the product and the inverse. A series is \emph{rational} if it belongs to the rational closure of $\agp$ in 
$\ags$.
\end{defn}

For $x\in \ulx$ let $\cA^x$ denote the $\cA$-bimodule generated by $x$, i.e., the set of homogeneous linear generalized polynomials in $x$ over $\cA$. More generally, let $\cA^w= \cA^{w_1}\cdots \cA^{w_{|w|}}$ for $w\in \gp$ (here we set $\cA^1=\cA$).

\begin{rem}\label{r:bimod}
For future reference we note that $\cA^w$ and $\cA^{\otimes |w|}$ are isomorphic as $\cA$-bimodules; in particular, as a 
$\cA$-bimodule, $\cA^w$ does not depend on the letters in $w$, but only on the length of $w$.
\end{rem}

\begin{defn}\label{d:recognizable}
A series $S$ is \emph{recognizable} if for some $n\in\N$ there exist $\cc\in \cA^{1\times n}$, $\bb\in \cA^{n\times 1}$ and 
$A^x\in (\cA^x)^{n\times n}$ for $x\in \ulx$ such that $[S,w]=\cc A^w\bb$ for all $w\in \gp$, where the notation
$$A^w = A^{w_1}\cdots A^{w_{|w|}} \in\cA^w$$
is used. In this case $(\cc,A,\bb)$ is called a \emph{(linear) representation of dimension $n$} of $S$.

Two representations $(\cc_1,A_1,\bb_1)$ and $(\cc_2,A_2,\bb_2)$ of dimension $n$ are \emph{similar} if there 
exists an invertible matrix $P\in\cA^{n\times n}$ (called a \emph{transition matrix}) such that
$$\cc_2=\cc_1P^{-1},\quad A_2=PA_1P^{-1},\quad \bb_2=P\bb_1.$$

\end{defn}

The only series with a zero-dimensional representation is the zero series $0$.


\subsection{Rationality equals recognizability}

In this subsection we prove that rational and recognizable series coincide using the notion of a stable module.

Let $L:\gp\to \End_{\cA-\cA}(\ags)$ be the map defined by
\begin{equation}\label{e:opL}
L_vS= \sum_{w\in \gp}[S,vw].
\end{equation}
We observe that $L_{uu'} S=L_{uu'}(L_u S)$ and $L_{uu'} (TS)=TL_{u'}S$ for $u,u'\in \gp$ and $T\in \cA^u$. It is also easy to check that the map $L_x$ for $x\in\ulx$ has derivative-like properties
\begin{equation}\label{e:L}
L_x(ST)=(L_xS)T+[S,1]L_xT,\quad L_xU^{-1}=-[U,1]^{-1}(L_xU)U^{-1}
\end{equation}
for $S,T,U\in \ags$ and $U$ invertible.

\begin{defn}
A left $\cA$-submodule $\cM\subseteq \ags$ is \emph{stable} if 
$L_x S\in \cA^x \cM$ for every $x\in \ulx$ and $S\in \cM$.
\end{defn}

\begin{prop}\label{p:stable}
A series $S\in \ags$ is recognizable if and only if it is contained in a stable finitely generated left $\cA$-submodule of $\ags$.
\end{prop}

\begin{proof}
($\Rightarrow$) Let $S$ be a recognizable series and $(\cc,A,\bb)$ its linear representation. For $i=1,\dots,n$ define $S_i\in\ags$ by $[S_i,w]=(A^w\bb)_i$ and let $\cM$ be the left $\cA$-submodule generated by the $S_i$. Then obviously $S\in \cM$. Since
$$[L_xS_i,xw]=[S_i,xw]=(A^{xw}\bb)_i=(A^xA^w\bb)_i=\sum_jA^x_{ij}[S_j,w]$$
for every $w$, we have
$$L_xS_i=\sum_jA^x_{i,j}S_j.$$
Therefore $\cM$ is stable.

($\Leftarrow$) Let $\cM$ be a stable $\cA$-submodule generated by finitely many elements $S_i$, and containing $S$. Then
$$S=\sum b_i S_i$$
for some $b_i\in \cA$. Furthermore, for every $x\in \ulx$ there exists $A^x\in (\cA^x)^{n\times n}$ such that
$$L_xS_i=\sum_j \cA^x_{ij}S_j.$$
By induction, it is easily proved that
$$L_vS_i=\sum_j A^v_{ij}S_j$$
holds for every $v$. Lastly, set $c_i=[S_i,1]$. Then
$$[S_i,w]=[L_wS_i,w]=\sum_j A^w_{ij}[S_j,1]=(A^w\bb)_i$$
and so $[S,w]=\cc A^w \bb$.
\end{proof}

A this point, Proposition \ref{p:stable} is just an intermediate step towards Theorem \ref{t:equiv}; however, we shall return to it again in Subsection \ref{subs:hankel}.

\begin{lem}\label{l:mat}
Let $M\in\left( \ags\right)^{n\times n}$ and suppose that the constant terms of all its entries are 0. Then $I-M$ is invertible and the entries of  $(I-M)^{-1}$ lie in the rational closure of $\F$ and the entries of $M$.
\end{lem}

\begin{proof}
We prove the statement by induction on $n$. The the case $n=1$ holds by characterization of invertible elements in $\ags$. Assume the lemma holds for all matrices of this kind and whose size is less than $n$. We decompose $M$ into four blocks
$$M=\begin{pmatrix}A&B\\C&D\end{pmatrix}$$
with square matrices $A$ and $D$. By induction hypothesis $I-A$ and $I-D$ are invertible. Moreover, the entries of matrices
$$A+B(I-D)^{-1}C,\quad D+C(I-A)^{-1}B$$
also have zero constant terms, so 
$$(I-A)-B(I-D)^{-1}C,\quad (I-D)-C(I-A)^{-1}B$$
are also invertible. By a known result about the inverse of a block matrix, see e.g. \cite[Subsection 0.7.3]{HJ}, $I-M$ is invertible and  its inverse equals
$$\begin{pmatrix}
((I-A)-B(I-D)^{-1}C)^{-1}& (I-A)^{-1}B(C(I-A)^{-1}B-(I-D))^{-1} \\
(C(I-A)^{-1}B-(I-D))^{-1}CA^{-1} & ((I-D)-C(I-A)^{-1}B)^{-1}
\end{pmatrix}.$$
Now it is also clear that the entries of this matrix lie in the rational closure of $\F$ and the entries of $M$.
\end{proof}

\begin{thm}\label{t:equiv}
A series is rational if and only if it is recognizable.
\end{thm}

\begin{proof}
($\Rightarrow$) For $P\in \agp$ and $w\in\supp P$ choose $a_{w,i}^{(j)}\in \cA$ such that
$$[P,w]=\sum_{i=1}^{n_w} a_{w,i}^{(0)}w_1a_{w,i}^{(1)}w_2\cdots w_{|w|}a_{w,i}^{(|w|)}.$$
Then
$$\{w_ja_{w,i}^{(j+1)}w_{j+1}\cdots w_{|w|}a_{w,i}^{(|w|)} \colon w\in\supp P,\, 1\le i\le n_w,\, j\ge1\}$$
is a stable finite set. Let $\cM$ be the left $\cA$-module generated by it; then $\cM$ is stable and $P\in \cM$. Thus $P$ is recognizable by Proposition \ref{p:stable}.

Let $S$ and $T$ be recognizable series; by Proposition \ref{p:stable} they belong to some stable finitely generated left $\cA$-modules $\cM$ and $\cN$. Left $\cA$-module $\cM+\cN$ is obviously stable and finitely generated. Also the left $\cA$-module $\cM T+\cN$ is finitely generated and stable by \eqref{e:L}. Since 
$S+T\in \cM+\cN$ and $ST\in \cM T+\cN$, the series $S+T$ and $ST$ are recognizable by Proposition \ref{p:stable}.

Assume $S$ is invertible and let $\cM$ be as in previous paragraph. The left 
$\cA$-module $\cA S^{-1} +\cM S^{-1}$ is finitely generated and stable, because 
$L_xS^{-1}=-[S,1]^{-1}(L_xS)S^{-1}$ and
$$L_x(S'S^{-1})=(L_xS'-[S',1][S,1]^{-1}L_xS)S^{-1}$$
for $S'\in \cM$ by \eqref{e:L}. Also $S^{-1}\in \cA S^{-1}+\cM S^{-1}$, so the series $S^{-1}$ is recognizable by Proposition \ref{p:stable}.

Therefore any rational series is recognizable by Definition \ref{d:rational}. 

($\Leftarrow$) Let $S$ be a recognizable series with representation $(\cc,A,\bb)$ of dimension $n$ and consider
$$M=\sum_{x\in \ulx}A^x\in \left(\agp\right)^{n\times n}.$$
This matrix satisfies the assumptions of Lemma \ref{l:mat}, so $I-M$ is invertible and its entries are rational series. Therefore
$$S=\sum_w \cc A^w\bb=\cc\left(\sum_wA^w\right)\bb=\cc\left(\sum_{k\ge0}M^k\right)\bb=\cc\left(I-M\right)^{-1}\bb$$
is a rational series.
\end{proof}

\begin{rem}\label{r:class}
If $\cA=\F$, then our terminology and results compare with classical representation theory of series as presented in \cite[Chapter 1]{BR}. For example, our notation and the notation in \cite[Chapter 1]{BR} are related via $[S,w]=(S,w)w$, and the role of $L_v$ is taken by $v^{-1}S:=\sum_{w}(S,vw)w$. Proposition \ref{p:stable} and Theorem \ref{t:equiv} and their proofs are natural analogs of \cite[Proposition 1.5.1 and Theorem 1.7.1]{BR}.
\end{rem}


\subsection{Arithmetics of representations}\label{subs:arith}

As already seen in the previous subsection, every recognizable series is rational by Theorem \ref{t:equiv}. In this subsection we start with some observations about elementary polynomials and then present concrete representations corresponding to the rational operations in Theorem \ref{t:arithmetic} following the ideas in \cite[Section 4]{BGM}. 

The constant $1$ has the representation $(1,0,1)$ of dimension 1. Let $a\in \cA$; then $x_i+a$ has the representation
\begin{equation}\label{a1}
\left(
\begin{pmatrix}1& a\end{pmatrix},
\begin{pmatrix}0&\delta_{ij}x_j\\ 0&0\end{pmatrix},
\begin{pmatrix}0\\ 1\end{pmatrix}
\right)
\end{equation}
of dimension 2, where $\delta_{ij}$ is the Kronecker's delta. If $S$ has the representation $(\cc,A,\bb)$, then $aS$ and $Sa$ have representations $(a\cc,A,\bb)$ 
and $(\cc,A,\bb a)$, respectively.

\begin{thm}\label{t:arithmetic}
For $i\in\{1,2\}$ let $S_i$ be a recognizable series with representation $(\cc_i,A_i,\bb_i)$ of dimension $n_i$ and $S$ be an invertible recognizable series with representation $(\cc,A,\bb)$ of dimension $n$. Then:
\begin{enumerate}[\rm(1)]
\item $S_1+S_2$ is recognizable with a representation
\begin{equation}\label{a2}
\left(
\begin{pmatrix}\cc_1 & \cc_2\end{pmatrix},
\begin{pmatrix}A_1&0\\0&A_2\end{pmatrix},
\begin{pmatrix}\bb_1\\ \bb_2\end{pmatrix}
\right)
\end{equation}
of dimension $n_1+n_2$;

\item $S_1S_2$ is recognizable with a representation
\begin{equation}\label{a3}
\left(
\begin{pmatrix}\cc_1 & \cc_1\bb_1\cc_2\end{pmatrix},
\begin{pmatrix}A_1&A_1\bb_1\cc_2 \\0&A_2\end{pmatrix},
\begin{pmatrix}0\\ \bb_2\end{pmatrix}
\right)
\end{equation}
of dimension $n_1+n_2$; 

\item $S^{-1}$ is recognizable with a representation
\begin{equation}\label{a4}
\left(
\begin{pmatrix}-a^{-1}\cc & a^{-1}\end{pmatrix},
\begin{pmatrix}A(I-\bb a^{-1}\cc)&A\bb a^{-1}\\0&0\end{pmatrix},
\begin{pmatrix}0\\ 1\end{pmatrix}
\right)
\end{equation}
of dimension $n+1$, where $a=[S,1]$.
\end{enumerate}
\end{thm}

\begin{proof}

(1) This is clear since
$$[S_1+S_2,w]=[S_1,w]+[S_2,w]=\cc_1 A_1^w\bb_1+\cc_2 A_2^w\bb_2=
\begin{pmatrix}\cc_1&\cc_2\end{pmatrix}\cdot
\begin{pmatrix}A_1^w&0\\0&A_2^w\end{pmatrix}\cdot
\begin{pmatrix}\bb_1\\ \bb_2\end{pmatrix}.$$

(2) If $w=w_1\cdots w_\ell$, let $M_j=A_1^{w_j}$, $N_j=A_2^{w_j}$ and $Q=\bb_1\cc_2$. Since
\[ \begin{split}
[S_1S_2,w]
&=\sum_{uv=w}[S_1,u][S_2,w]
=\cc_1\bb_1\cc_2 A_2^w\bb_2+\cc_1\left(\sum_{uv=w,|u|>0} A_1^u \bb_1\cc_2 A_2^v \right)\bb_2 \\
&=(\cc_1 Q)N_1\cdots N_\ell c_2+\cc_1\left(
\sum_{k=1}^\ell M_1\cdots M_k Q N_{N+1}\cdots N_\ell \right)\bb_2,
\end{split} \]
it is enough to prove the equality
$$\prod_{j=1}^\ell \begin{pmatrix}M_j&M_jQ\\0&N_j\end{pmatrix}=
\begin{pmatrix}\prod_{j=1}^\ell M_j&
\sum_{k=1}^\ell \left(\prod_{j=1}^k M_j\right)Q\left(\prod_{j=k+1}^\ell N_j\right)
\\0&\prod_{j=1}^\ell N_j\end{pmatrix}$$
and this can be easily done by induction on $\ell$.

(3) If $w=w_1\cdots w_\ell$, let $M_j=A^{w_j}$ and $Q=\bb a^{-1}\cc$. The statement is proved by induction on $\ell$. It obviously holds for $\ell=0$, so let $\ell\ge 1$. Note that representation \eqref{a4} yields a series $T$ with
$$[T,w]=-a^{-1}\cc A^{w_1}(I-\bb a^{-1}\cc)\cdots A^{w_{\ell-1}}(I-\bb a^{-1}\cc)
A^{w_\ell}\bb a^{-1}.$$
By the inductive step we have
\[ \begin{split}
[S^{-1},w]
&=-\sum_{uv=w,u\neq 1} a^{-1}[S,u][S^{-1},v] \\
&=-\sum_{uv=w,u\neq 1} a^{-1}[S,u][T,v] \\
&=-a^{-1}\cc\left(M_1\cdots M_\ell
-\sum_{i=1}^{l-1} M_1\cdots M_i Q M_{i+1}(I-Q)\cdots M_{\ell-1}(I-Q)M_\ell \right)\bb a^{-1} \\
&=-a^{-1}\cc M_1(I-Q)\cdots M_{\ell-1}(I-Q)M_\ell\bb a^{-1} \\
&=[T,w]
\end{split} \]
and thus the statement holds.
\end{proof}


\subsection{Evaluations of rational series}\label{subs:eval}

In this subsection let $\cA=M_m(\F)$. A combination of rational operations and generalized polynomials that represents an element of $M_m(\F)\gs$ is called a \emph{generalized rational expression over $M_m(\F)$} or shortly a gr expression. Therefore each gr expression is a rational series, but a rational series can be written as a gr expression in different ways. For example, $x^{-1}$ for $x\in \ulx$ is not a gr expression since $x\in M_m(\F)\gs$ is not invertible in $M_m(\F)\gs$; on the other hand, $(1-x)^{-1}$ and $1+(1-x)^{-1}-1$ are different gr expressions but yield the same rational series, namely $\sum_{k\ge 0}x^k$.

Let $s\in\N$ and consider the inclusion
\begin{equation}\label{e:funinc}
\iota:M_m(\F)\to M_{ms}(\F),\quad a\mapsto I_s\otimes a,
\end{equation}
where $\otimes$ is the Kronecker product. We can now evaluate generalized polynomials over $M_m(\F)$ in $M_{ms}(\F)$ using $\iota$ in a natural way. Of course, one could choose a different inclusion $M_m(\F)\to M_{ms}(\F)$ to define these evaluations. However, all such inclusions are conjugate by the Skolem-Noether theorem \cite[Theorem 3.1.2]{Row} and the corresponding evaluations thus exhibit the same behavior. Therefore we establish the convention to always use $\iota$.

Furthermore, we can inductively define the domain $\dom_{ms} S\subseteq M_{ms}(\F)^g$ and the evaluation $S(\ulq)$ at $\ulq\in\dom_{ms}S$ for any gr expression $S$, similarly as in Section \ref{sec2}. By the definition, every gr expression $S$ is defined in 0 and $S(0)=[S,1]$. We also define
$$\dom S=\cup_{s\in\N} \dom_{ms}S.$$

Since
$$\sum_w\cc A^w\bb=\cc\left(I-\sum_x A^x\right)^{-1}\bb$$
by the proof of Theorem \ref{t:equiv}, we can also consider an evaluation of the linear representation by evaluating the entries of matrix $I-\sum_x A^x$ at the points in $M_{ms}(\F)^g$ where the resulting matrix is invertible. Of course, each linear representation yields at least one gr expression because the entries of the inverse of $I-\sum_x A^x$ can be calculated as in the proof of Lemma \ref{l:mat}. This gr expression might have a smaller domain, but where defined, its evaluation equals the evaluation of the linear representation.

For $1\le k\le g$ let $T_k$ be the generic $ms\times ms$ matrices. If a gr expression $S$ is evaluated at $(T_1,\dots, T_g)$, we get an element of $M_{ms}(\F(t_{ij}^{(k)}))$. But $S$ can be also evaluated at $(T_1,\dots, T_g)$ as a series resulting in an element of $M_m(\F\langle\!\langle t_{ij}^{(k)}\rangle\!\rangle)$. If we expand the commutative rational functions of the former matrix in a formal commutative power series about 0, we get the latter matrix.

Since $\F$ is infinite, a commutative rational function defined at 0 evaluates to 0 wherever defined if and only if its formal power series about 0 is the zero series. Thus the previous discussion implies that the evaluation of a gr expression $S$ at the generic tuple depends only on $S$ as the series. Therefore all gr expressions and linear representations of a given rational series yield the same evaluations wherever defined.
 
\begin{defn}
Let $S$ be a rational series. The matrix $S(T_1,\dots,T_g)\in M_{ms}(\F(t_{ij}^{(k)}))$ is called \emph{the generic evaluation of a series $S$ at the generic matrices of size $ms$}, and the intersection of the domains of its entries is denoted $\edom_{ms} S\subseteq M_{ms}(\F)^g$. Let
$$\edomstn{ms}S=\left\{\ulq\in \edom_{ms}S\colon
\overbrace{\ulq\oplus\dots\oplus\ulq}^k\in\edom_{kms}S \text{ for all } k\in\N\right\};$$
the set
$$\edomst S=\bigcup_{s\in\N} \edomstn{ms}S$$
is called the \emph{stable extended domain} of $S$.
\end{defn}

Although not addressed here, one can consider convergence of a generalized series with respect to norm topology if $\F\in\{\R,\C\}$; from this analytical point of view, evaluations of generalized series can be studied in a much wider sense. For an elaborate exposition we refer to \cite[Chapter 8]{KVV3} or \cite{Voi1, Voi2}.

We continue with a short review of the connection between nc polynomials over $\F$ and generalized polynomials over $M_m(\F)$. Following \cite[Section 1.7]{Co}, a \emph{$m\times m$ matrix ring} is a ring $R$ with $m^2$ elements $E_{ij}\in R$, called the matrix units, satisfying
$$E_{ij}E_{kl}=\delta_{jk}E_{il},\quad \sum_{i=1}^mE_{ii}=1,$$
and a homomorphism of $m\times m$ matrix rings is a homomorphism of rings that preserves the matrix units. For example, $M_m(\F)\gp$ and $M_{ms}(\F)$ are $m\times m$ matrix rings with the matrix units $e_{ij}$ and $e_{ij}\otimes I_s$, respectively, where $e_{ij}$ are the standard matrix units in $M_m(\F)$. Moreover, we have
\begin{equation}\label{e:mrfunctor}
\begin{split}
\Hom_{M_m}(M_m(\F)\gp,M_{ms}(\F))
&\cong \Hom_{M_m}(M_m(\fW_m(\F\gp)),M_m(M_s(\F)))\\
&\cong \Hom(\fW_m(\F\gp),M_s(\F)),
\end{split}
\end{equation}
where $\fW_m$ denotes the matrix reduction functor (see \cite[Section 1.7]{Co}). For a free algebra we have
$$\fW_m(\F\gp)=\F\gpfx,$$
where
$$\fX=\{\fx_{ij}^{(k)}\colon 1\le i,j\le m,\ 1\le k\le g\}$$
is an alphabet. The first isomorphism in \eqref{e:mrfunctor} follows from the isomorphism
\begin{equation}
\psi:M_m(\F)\langle \ulx\rangle\to M_m(\F\gpfx),\quad 
e_{i\imath}x_ke_{\jmath j}\mapsto \fx_{\imath\jmath}^{(k)} e_{ij}.
\end{equation}
The second isomorphism in \eqref{e:mrfunctor} is a consequence of the equivalence between the category of rings and the category of $m\times m$ matrix rings \cite[Theorem 1.7.1]{Co}.

\begin{prop}\label{p:GPI}
If $f\in M_m(\F)\gp$ is of degree $h$ and vanishes on matrices of size 
$m\lceil\tfrac{h+1}{2}\rceil$, then $f=0$.
\end{prop}

\begin{proof}
Set $s=\lceil\tfrac{h+1}{2}\rceil$ and write
$$f=\sum_{|w|\le h} \sum_i a_{w,i}^{(0)}w_1a_{w,i}^{(1)}w_2\cdots w_{|w|}a_{w,i}^{(|w|)}$$
for $a_{w,i}^{(\ell)}\in M_m(\F)$. Then
$$\sum_{|w|\le h} \sum_i \left(I_s\otimes a_{w,i}^{(0)}\right)p_{w_1}\left(I_s\otimes a_{w,i}^{(1)}\right)p_{w_2}\cdots p_{w_{|w|}}\left(I_s\otimes a_{w,i}^{(|w|)}\right)=0$$
for all $p\in M_{ms}(\F)^g$ by the assumption. By the Skolem-Noether theorem \cite[Theorem 3.1.2]{Row}, there exists $q_0\in M_{ms}(\F)$ such that $I_s\otimes a=q_0(a\otimes I_s)q_0^{-1}$ for every $a\in M_m(\F)$. Therefore
$$\sum_{|w|\le h} \sum_i \left(a_{w,i}^{(0)}\otimes I_s\right)p_{w_1}\left(a_{w,i}^{(1)}\otimes I_s\right)p_{w_2}\cdots p_{w_{|w|}}\left(a_{w,i}^{(|w|)}\otimes I_s\right)=0$$
for all $p\in M_{ms}(\F)^g$. Hence $\Phi(f)=0$ for every $\Phi\in\Hom_{M_m}(M_m(\F)\gp,M_{ms}(\F))$ and thus $\Psi(\psi(f))=0$ for every $\Psi\in \Hom(\fW_m(\F\gp),M_s(\F))$. Since isomorphism $\psi$ preserves polynomial degree and there are no nonzero polynomial identities on $M_s(\F)$ of degree less than $2s\ge h$ (see e.g. \cite[Lemma 1.4.3]{Row}), we have $\psi(f)=0$ and finally $f=0$.
\end{proof}

\begin{prop}\label{p:zeroseries}
If a gr expression $S$ over $M_m(\F)$ vanishes on its domain, then $S$ is the zero series.
\end{prop}

\begin{proof}
Let $s\in\N$ be arbitrary and let $\ulT=(T_1,\dots, T_g)$ be the tuple of generic $ms\times ms$ matrices. Then $S(\ulT)$ is a matrix of commutative rational functions and the matrix of their expansions about $p$ is
$$M(\ulT)=\sum_w[S,w](\ulT).$$
The formal differentiation of commutative power series yields
\begin{equation}\label{e:der}
\frac{\der}{\der t^h} M\big(t\ulT\big)\Big|_{t=0}
=h!\sum_{|w|=h}[S,w](\ulT).
\end{equation}
As $S$ vanishes on $\dom_{ms} S$, we have $M(\ulT)=0$; therefore the left-hand side of \eqref{e:der} equals 0 for every $h$, hence the same holds for the right-hand side of \eqref{e:der}. Since $\sum_{|w|=h}[S,w]$ is a generalized polynomial of degree $h$ and $s$ is arbitrary, we have $[S,w]=0$ for all $w\in\gp$ by Proposition \ref{p:GPI}.
\end{proof}

Aside from \cite[Chapter 8]{KVV3}, comparable analytic versions of Proposition \ref{p:zeroseries} can also be found in \cite[Section 13]{Voi2} and \cite[Section 4]{KS}.



\section{Minimization}\label{sec4}

Up until now the algebra $\cA$ was quite general (except in Subsection \ref{subs:eval}). Since our motivation comes from studying nc rational functions, which are most often studied through their evaluations on division rings or matrices, we will restrict $\cA$ to being a simple Artinian $\F$-algebra. Thus the first subsection revises some of the properties of these rings.

Later on, we define three special types of representations: reducible, minimal and totally reducible representations, which have parallels in the classical theory (compare with controllability, observability and minimality in \cite[Sections 5, 6, 9]{BGM}, \cite[Subsection 4.1.1]{HMV} or \cite[Section 2.2]{BR}). The main results are Proposition \ref{p:red}, which provide us with a method of finding reduced representations, and Theorem \ref{t:redmin}, which asserts that reduced representations are not far away from minimal and that totally reduced representations are unique up to similarity.

Finally, we relate the dimension of minimal representations to the rank of the Hankel module, which is a natural substitute for the notion of a Hankel matrix.

For the rest of the paper the standard matrix units in $M_m(\F)$ are denoted by $e_{ij}$.


\subsection{Tools from the Artin-Wedderburn theory}\label{subs:sa}

Let $\cA$ be a simple Artinian ring with center $\F$. We state some well-known facts about $\cA$ and its left (right) modules that can be found in \cite[Sections 1.2 and 1.3]{La} and are part of the classical Artin-Wedderburn theory.
\begin{enumerate}
\item $\cA\cong M_m(D)$ for a unique $m\in\N$ and up-to-isomorphism unique central division $\F$-algebra $D$.
\item Every submodule of a left (right) $\cA$-module is a direct summand in it.
\item Every finitely generated left (right) $\cA$-module is up to an isomorphism of $\cA$-modules uniquely determined by its dimension over $D$, which is divisible by $m$.
\item In particular, a finitely generated left (right) $\cA$-module is free if and only if its dimension over $D$ is divisible by $m^2$, and is torsion if and only if its dimension over $D$ is less than $m^2$.
\item Every left (right) $\cA$-linearly independent subset $S$ of $\cA^n$ can be extended to a left (right) $\cA$-basis of $\cA^n$.
\end{enumerate}

By a torsion module we mean a left (right) $\cA$-module $\cM$ such that for every $\vv\in\cM$ there exists $a\in\cA\setminus\{0\}$ satisfying $a \vv=0$ ($\vv a=0$).

Let $\uly=\{y_k\colon k\in\N\cup\{0\}\}$ 
be an infinite alphabet. The following result can also be found 
in \cite[Corollary 7.2.13]{Row}. For the sake of self-containment we provide a short proof.

\begin{lem}\label{l:difvar}
If $f\in\cA^{y_0}\cdots\cA^{y_k}$ vanishes on $\cA$, then $f=0$.
\end{lem}

\begin{proof}
We prove the claim by induction on $k$.

For the induction basis let $k=0$. Assume $f\in\cA^{y_0}$ vanishes on $\cA$ and consider the homomorphism of $\F$-algebras
\begin{equation}\label{e:map}\cA\otimes_{\F}\cA^{op}\to \End_{\F}(A),\quad a\otimes b\mapsto L_aR_b,\end{equation}
where $L_a$ and $R_b$ are multiplications by $a$ on the left and by $b$ on the right, respectively. It is 
a classical result (e.g. \cite[Theorem 3.1]{La}) that the $\F$-algebra $\cA\otimes_{\F}\cA^{op}$ is simple. Therefore the 
map \eqref{e:map} is injective. As noted in Remark \ref{r:bimod}, $\cA\otimes_{\F}\cA^{op}\cong \cA^{y_0}$,
so $f=0$.

For the induction step assume that the claim is established for $k-1$. Suppose $f\in\cA^{y_0}\cdots\cA^{y_k}$ vanishes on $\cA$. We can write it in 
the form 
$$f=\sum_i a_iy_0f_i,$$
where $a_i\in \cA$ are $\F$-linearly independent and $f_i\in \cA^{y_1}\cdots\cA^{y_k}$. Let 
$b_1,\dots b_k\in \cA$ be arbitrary and $\tilde{f}=f(y_0,b_1,\dots,b_k)$. By the basis of induction we have $\tilde{f}=0$. Since $a_i$ are $\F$-linearly independent, we have 
$f_i(b_1,\dots,b_k)=0$ for all $i$. Since $b_j$ were arbitrary, 
we have $f_i=0$ by the induction hypothesis and therefore $f=0$.
\end{proof}

We will frequently apply Lemma \ref{l:difvar} to $\cA^w$ for $w\in \gp$ using the isomorphism $\cA^w\cong \cA^{y_1}\cdots \cA^{y_{|w|}}$ from Remark \ref{r:bimod}. For notational simplicity we establish the following convention: if $f\in\cA^w$, then $f(a_1,\dots,a_g)$ denotes the usual evaluation of $f$ at a tuple $(a_1,\dots,a_g)\in\cA^g$, and $f[a_1,\dots, a_{|w|}]$ denotes the evaluation of $f$ as an element of $\cA^{y_1}\cdots \cA^{y_{|w|}}$ at a tuple $(a_1,\dots, a_{|w|})\in\cA^{|w|}$.


\subsection{Reduced, minimal and totally reduced representations}\label{subs:rmtr}

In this subsection we consider special kinds of representations that differ in availability and utility.

Let $S$ be a recognizable series with a linear representation $(\cc,A,\bb)$ of dimension $n$. For $N\in\N\cup\{0\}$ let
$$\cU_N^L=\{\uu\in \cA^{1\times n}\colon \uu A^w\bb=0\ \forall |w|\le N\},\quad
\cU_N^R=\{\vv\in \cA^{n\times 1}\colon \cc A^w\vv=0\ \forall |w|\le N\}.$$
Then $\cU_N^L$ is left $\cA$-module, $\cU_N^R$ is right $\cA$-module, $\cU_N^L\supseteq \cU_{N+1}^L$ and 
$\cU_N^R\supseteq \cU_{N+1}^R$. The modules
$$\cU_\infty^L=\bigcap_{N\in\N}\cU_N^L \quad \text{and}\quad \cU_\infty^R=\bigcap_{N\in\N}\cU_N^R$$
are called the \emph{left} and \emph{right obstruction module} of $(\cc,A,\bb)$, respectively.

\begin{defn}\label{d:red}
The representation $(\cc,A,\bb)$ of $S$ is:
\begin{enumerate}[(a)]
\item \emph{reduced} if its obstruction modules are torsion $\cA$-modules;
\item \emph{minimal} if its dimension is minimal among all the representations of $S$.
\item \emph{totally reduced} if its obstruction modules are trivial.
\end{enumerate}
\end{defn}

By definition, every totally reduced representation is reduced. In classical representation/realization theory, both notions are equivalent to minimality; moreover, all minimal representations of a given series are similar (\cite[Theorem 9.1 and Theorem 8.2]{BGM} or \cite[Proposition 2.2.1 and Theorem 2.2.4]{BR}). However, this is not true in our setting. Later we prove 
$(c)\Rightarrow (b)\Rightarrow (a)$ (Corollary \ref{c:cba}). The following examples show that these implications cannot be reversed.

\begin{exa}
Minimal representations are in general neither pairwise similar nor totally reduced. If $\cA=M_2(\F)$, then the series 
$\sum_{i\ge0}e_{11}(xe_{11})^i$ has non-similar minimal representations, e.g. $(1,e_{11}x,e_{11})$ and $(e_{11},e_{11}xe_{11},1)$.
\end{exa}

\begin{exa}
A reduced representation is not necessarily minimal. Again let $\cA=M_2(\F)$. Previously we specified that the zero series has a representation of dimension 0, but it also has a reduced representation of dimension 1, namely $(e_{11},e_{22}x,e_{22})$. For a less contrived example, one can consider the series $S=\sum_{i\ge0}(e_{11}x)^i$. It has a minimal realization $(1,e_{11}x,1)$ of dimension 1, but it also has a reduced representation
$$\left(
\begin{pmatrix}e_{11}& 1\end{pmatrix},
\begin{pmatrix}e_{22}x&0\\ 0&e_{11}x\end{pmatrix},
\begin{pmatrix}e_{22}\\ 1\end{pmatrix}
\right)$$
of dimension 2: indeed, elements in $\cU_1^L$ and $\cU_1^R$ are of the form
$$\begin{pmatrix}\alpha e_{11}+\beta e_{21}& 0\end{pmatrix}\qquad \text{and} \qquad
\begin{pmatrix}\alpha e_{21}+\beta e_{22}\\ 0\end{pmatrix},$$
respectively.
\end{exa}

From here on $\cA$ is a central simple Artinian $\F$-algebra that is isomorphic to $M_m(D)$, where $D$ is a division ring. We start with some observations about the obstruction modules $\cU_N^L$; appropriate right-hand versions hold for $\cU_N^R$.

\begin{lem}\label{l:stop}
If $\cU_N^L=\cU_{N+1}^L$, then $\cU_{N+1}^L=\cU_{N+2}^L$.
\end{lem}

\begin{proof}
If $\uu\notin \cU_{N+2}^L$, then $\uu A^{xw} \bb \neq0$ for some $x\in \ulx$ and $|w|=k\le N+1$. Let $f\in\cA^{xw}$ 
be the nonzero entry of $\uu A^{xw} \bb$. By Remark \ref{r:bimod}, $f[y_0,\dots,y_k] \neq0$. By 
Lemma \ref{l:difvar}, there exists $b\in \cA$ such that $f[b,y_1,\dots,y_k] \neq0$. Again by 
Remark \ref{r:bimod}, $f[b,w_1,\dots,w_k]\neq0$ and so $\uu A^x(b)A^w\bb\neq0$. Therefore 
$\uu A^x(b)\notin \cU_{N+1}^L=\cU_N^L$, hence $\uu A^x(b)A^{w'}\bb\neq0$ for some 
$|w'|=k'\le N$. Let $h\in\cA^{y_0w'}$ be such entry of $\uu A^x(y_0)A^{w'}\bb$ that $h[b,w'_1,\dots,w'_{k'}]\neq0$. Thus 
$h\neq0$ and also $h[x,w'_1,\dots,w'_{k'}]\neq0$ by Remark \ref{r:bimod}. Therefore $\uu A^xA^{w'}\bb\neq0$ and $\uu\notin \cU_{N+1}^L$.
\end{proof}

\begin{lem}\label{l:termination}
If $\cA\cong M_m(D)$ and the dimension of $(\cc,A,\bb)$ is $n$, then $\cU_\infty^L=\cU_{mn-1}^L$.
\end{lem}

\begin{proof}
The statement trivially holds for $\bb=0$, so we assume $\bb\neq0$. As stated in Subsection 
\ref{subs:sa}, the dimension of every left $\cA$-module as a vector space over $D$ is divisible by $m$. Since the dimension of the vector space $\cA^{1\times n}$ over $D$ is $m^2n$ and $\bb\neq0$, the descending chain of left $\cA$-modules $\{\cU_N^L\}_{N\in\N}$ stops by 
Lemma \ref{l:stop} and
\[
\cA^n\supseteq \cU_0^L\supseteq \cU_1^L\supseteq \cdots \supseteq \cU_{mn-1}^L=\cU_{mn}^L=\cdots.
\qedhere \]
\end{proof}

\begin{prop}\label{p:red}
Every recognizable series has a reduced representation.
\end{prop}

\begin{proof}
Fix some representation $(\cc,A,\bb)$ of a series $S$. We will describe how to obtain a reduced representation of $S$ from $(\cc,A,\bb)$.

Let $\{\uu_{k+1},\dots,\uu_n\}\subset \cU_\infty^L$ be a maximal left $\cA$-linearly independent set. As stated in Subsection \ref{subs:sa}, it can be extended to a basis $\{\uu_1,\dots,\uu_n\}$ of $\cA^{1\times n}$. Let $U$ be a matrix with rows $\uu_i$. 
Then $(\cc U^{-1},UAU^{-1},U\bb)$ is a similar representation of $S$ corresponding to the basis change from the standard basis 
of $\cA^{1\times n}$ to $\{\uu_1,\dots,\uu_n\}$. Let 
$$\cc U^{-1}=\begin{pmatrix}\cc_1&\cc_2\end{pmatrix},\quad
U A U^{-1}=\begin{pmatrix}A_{11}& A_{12} \\ A_{21} & A_{22}\end{pmatrix},\quad
U\bb=\begin{pmatrix}\bb_1 \\ \bb_2\end{pmatrix}$$
be the block decompositions with regard to $\{\uu_1,\dots,\uu_k\}\cup\{\uu_{k+1},\dots,\uu_n\}$. We 
claim that
\begin{equation}\label{e:red1}
\begin{pmatrix}\vv_1&\vv_2\end{pmatrix}
\begin{pmatrix}A_{11}& A_{12} \\ A_{21} & A_{22}\end{pmatrix}^w
\begin{pmatrix}\bb_1 \\ \bb_2\end{pmatrix} =
\begin{pmatrix}\vv_1&0\end{pmatrix}
\begin{pmatrix}A_{11}& 0 \\ 0 & 0\end{pmatrix}^w
\begin{pmatrix}\bb_1 \\ 0\end{pmatrix}
\end{equation}
for all $\vv\in\cA^{1\times n}$ and $w\in \gp$. Indeed, because 
$\{\uu_{k+1},\dots,\uu_n\}\subset \cU_\infty^L$, we have
$$\begin{pmatrix}\vv_1&\vv_2\end{pmatrix}\begin{pmatrix}\bb_1 \\ \bb_2\end{pmatrix}
=\begin{pmatrix}\vv_1&0\end{pmatrix}\begin{pmatrix}\bb_1 \\ 0\end{pmatrix}$$
and then
\begin{align*}
\begin{pmatrix}\vv_1&\vv_2\end{pmatrix}
\begin{pmatrix}A_{11}& A_{12} \\ A_{21} & A_{22}\end{pmatrix}^{xw}
\begin{pmatrix}\bb_1 \\ \bb_2\end{pmatrix}
=&\begin{pmatrix}\vv_1&0\end{pmatrix}
\begin{pmatrix}A_{11}& A_{12} \\ A_{21} & A_{22}\end{pmatrix}^{xw}
\begin{pmatrix}\bb_1 \\ \bb_2\end{pmatrix}\\
=&\begin{pmatrix}\vv_1A_{11}^x&\vv_1A_{12}^x\end{pmatrix}
\begin{pmatrix}A_{11}& A_{12} \\ A_{21} & A_{22}\end{pmatrix}^w
\begin{pmatrix}\bb_1 \\ \bb_2\end{pmatrix}\\
\eqind &\begin{pmatrix}\vv_1A_{11}^x&0\end{pmatrix}
\begin{pmatrix}A_{11}& 0 \\ 0 & 0\end{pmatrix}^w
\begin{pmatrix}\bb_1 \\ 0\end{pmatrix}\\
=&\begin{pmatrix}\vv_1&0\end{pmatrix}
\begin{pmatrix}A_{11}& 0 \\ 0 & 0\end{pmatrix}^{xw}
\begin{pmatrix}\bb_1 \\ 0\end{pmatrix}
\end{align*}
by induction on $|w|$. Setting $\cc'=\cc_1$, $A'=A_{11}$ and $\bb'=\bb_1$, we get a representation $(\cc',A',\bb')$ of $S$ by \eqref{e:red1}. The $\cA$-module $\cU'^L_\infty$ is torsion by construction since it is isomorphic to $\cU_\infty^L$ modulo a maximal free submodule.

Next, we apply the right version of the upper procedure: we pick a maximal right linearly 
independent set in $\cU'^R_\infty$, extend it to a basis of $\cA^{n\times 1}$, transform 
$(\cc',A',\bb')$ with respect to the new basis, and get a representation $(\cc'',A'',\bb'')$ of 
$S$ with torsion module $\cU''^R_\infty$ by taking the suitable blocks.

Because we have no information about $\cU''^L_\infty$, we repeat the cycle of previous two steps. We continue until one of the steps does not change the dimension of the representation; therefore our procedure finishes after at most $n$ steps, where $n$ is the dimension of the starting representation. It is clear that $\cU_\infty^L$ and $\cU_\infty^R$ corresponding to the final representation are torsion modules.
\end{proof}

\begin{rem}
The proof of Proposition \ref{p:red} can serve as an algorithm for reducing representations. If $\cA=M_m(\F)$, it can be implemented in a purely linear-algebraic way without any other ring-theoretical traits. The implementation is mostly straightforward, although a few notes are in order.
\begin{itemize}
\item It is convenient to treat the polynomials from $\cA^w$ as elements of $\cA^{\otimes|w|}\cong M_{m^{|w|}}(\F)$, because the use of tensors makes it easier to apply linear algebra for computing the modules $\cU_N^L$.
\item Determining the obstruction modules is computationally expensive. We start by computing the modules $\cU_0^L$, $\cU_1^L$, etc., and we stop according to Lemmata \ref{l:stop} and 
\ref{l:termination}.
\item A maximal $\cA$-linearly independent subset of $\cU_\infty^L$ can be found using \cite[Lemma 8]{CIK}.
\end{itemize}
\end{rem}

\begin{exa}
As mentioned in the proof of Proposition \ref{p:red}, in general there is no guarantee that the reduction on one side of the representation will not tarnish the other side. For example, consider the representation
$$\left(
\begin{pmatrix}e_{11}& e_{21} & 0\end{pmatrix},
\begin{pmatrix}
x e_{22}& x e_{22}& 0\\
0 &e_{22} x e_{22}& e_{11}x\\
0 & 0 & x
\end{pmatrix},
\begin{pmatrix}0\\0\\ e_{11}\end{pmatrix}
\right)$$
over $M_2(\F)$. One can check that $\cU_\infty^R=0$ and $\cU_\infty^L$ has a maximal free submodule of rank 1 generated by $(1, 0, 0)$. After the left-hand side reduction with respect to this vector we get a nontrivial 
$\cU'^R_\infty$.
\end{exa}

\begin{lem}\label{l:redfree}
Let $(\cc,A,\bb)$ be a representation and assume $\cU_\infty^L$ is a free module. Then $\cU'^R_\infty$ from the first step of algorithm in Proposition \ref{p:red} embeds into $\cU_\infty^R$.
\end{lem}

\begin{proof}
Assume the notation of the proof of Proposition \ref{p:red}. Note that $\{\uu_{k+1},\dots,\uu_n\}$ is now a basis for $\cU_\infty^L$; thus we have $A_{21}=0$.

Now we can prove a refined version of \eqref{e:red1}, namely
\begin{equation}\label{e:red2}
\begin{pmatrix}\vv_1&\vv_2\end{pmatrix}
\begin{pmatrix}A_{11}& A_{12} \\ 0 & A_{22}\end{pmatrix}^w
\begin{pmatrix}\vv_3 \\ \bb_2\end{pmatrix} =
\begin{pmatrix}\vv_1&0\end{pmatrix}
\begin{pmatrix}A_{11}& 0 \\ 0 & 0\end{pmatrix}^w
\begin{pmatrix}\vv_3 \\ 0\end{pmatrix}
\end{equation}
for all $\vv_1\in\cA^{1\times k}$, $\vv_2\in\cA^{1\times (n-k)}$, $\vv_3\in\cA^{k\times 1}$ 
and $w\in \gp$. We have
$$\begin{pmatrix}\vv_1&\vv_2\end{pmatrix}\begin{pmatrix}\vv_3 \\ \bb_2\end{pmatrix}
=\vv_1\vv_3+\begin{pmatrix}0&\vv_2\end{pmatrix}\begin{pmatrix}\bb_1\\ \bb_2\end{pmatrix}
=\begin{pmatrix}\vv_1&0\end{pmatrix}\begin{pmatrix}\vv_3 \\ 0\end{pmatrix}$$
and then
\begin{align*}
\begin{pmatrix}\vv_1&\vv_2\end{pmatrix}
\begin{pmatrix}A_{11}& A_{12} \\ 0 & A_{22}\end{pmatrix}^{xw}
\begin{pmatrix}\vv_3 \\ \bb_2\end{pmatrix}
=&\begin{pmatrix}\vv_1A_{11}^x&\vv_1A_{12}^x+\vv_2A_{22}^x\end{pmatrix}
\begin{pmatrix}A_{11}& A_{12} \\ 0 & A_{22}\end{pmatrix}^w
\begin{pmatrix}\vv_3 \\ \bb_2\end{pmatrix}\\
\eqind & \begin{pmatrix}\vv_1A_{11}^x&0\end{pmatrix}
\begin{pmatrix}A_{11}& 0 \\ 0 & 0\end{pmatrix}^w
\begin{pmatrix}\vv_3 \\ 0\end{pmatrix}\\
=&\begin{pmatrix}\vv_1&0\end{pmatrix}
\begin{pmatrix}A_{11}& 0 \\ 0 & 0\end{pmatrix}^{xw}
\begin{pmatrix}\vv_3 \\ 0\end{pmatrix}
\end{align*}
by induction. Setting $\vv_1=\cc_1$ and $\vv_2=\cc_2$ we get $\bb_2\in U\cU_\infty^R$ for $\vv_3=0$ and then 
\[
\vv_3\in \cU_\infty'^R \ \iff \begin{pmatrix}\vv_3 \\ 0\end{pmatrix}\in U\cU_\infty^R.
\qedhere \]
\end{proof}

\begin{cor}\label{c:redfree}
If $\cU_\infty^L$ and $\cU'^R_\infty$ are free modules, then the algorithm of Proposition \ref{p:red} yields a totally reduced representation after the first cycle.
\end{cor}

Our next goal is to show that reduced representations are not far from minimal ones. First we need the following lemma.

\begin{lem}\label{l:dimbound}
Let $(b,A,c)$ be a representation of dimension $n$ and define left $\cA$-modules
\begin{equation}\label{e:W0W}
\cW_0=\lin_{\cA}\{\cc A^w[a_1,\dots, a_{|w|}]\colon w\in \gp, \ a_i\in\cA\},
\qquad \cW=\cW_0/(\cW_0\cap \cU_\infty^L).
\end{equation}
Then the dimension of $\cW_0$ as a left vector space over $D$ is $m^2n-\dim_D\cU_\infty^R$ and the dimension of $\cW$ as a left vector space over $D$ is at least
$m^2n-(\dim_D\cU_\infty^R+\dim_D\cU_\infty^L)$.
\end{lem}

\begin{proof}
First we claim that
\begin{equation}\label{e:314}
\cU_\infty^R=\{\uu\in\cA^{n\times 1}\colon \vv\uu=0\ \forall \vv\in\cW_0 \}.
\end{equation}
Indeed, the inclusion $\subseteq$ is obvious and $\supseteq$ follows from Lemma \ref{l:difvar}.

It is enough to prove the first part of the lemma since the second part then follows from $\dim_D\cW\ge \dim_D\cW_0-\dim_D\cU_\infty^L$.

Let $\dim_D\cW_0=m^2i+mj$ for $j<m$. By Subsection \ref{subs:sa}, we can choose a basis of $\cA^{1\times n}$ in such way that $\cW_0$ equals $\cA^{1\times i}\times I \times 0^{1\times (n-i-1)}$, where $I\subset\cA$ is the right ideal of matrices whose first $m-j$ rows are zero. By \eqref{e:314}, $\cU_\infty^R$ equals  $0^{i\times 1}\times J \times \cA^{(n-i-1)\times 1}$, where $J\subset \cA$ is the left ideal of matrices whose last $j$ columns are zero. Therefore $\dim_D\cU_\infty^R=m^2(n-i-1)+m(m-j)$.
\end{proof}

\begin{thm}\label{t:redmin}
Let $S$ be a recognizable series.
\begin{enumerate}
\item The difference in dimensions of two reduced representations of $S$ is at most 1.
\item A representation is minimal if $\delta:=\dim_D\cU_\infty^L+\dim_D\cU_\infty^R<m^2$.
\item Totally reduced representations are similar via a unique transition matrix.
\end{enumerate}
\end{thm}

\begin{proof}
Let $(\cc,A,\bb)$ and $(\cc',A',\bb')$ be two representations for $S$ of dimensions $n$ and $n'$, and let $\cW$ and $\cW'$ be the corresponding left $\cA$-modules as in \eqref{e:W0W} from 
Lemma \ref{l:dimbound}. We want to show that the rule
\begin{equation}\label{e:isomap}
\cc A^w[a_1,\dots, a_{|w|}]\mapsto \cc' A'^w[a_1,\dots, a_{|w|}]
\end{equation}
defines a bijective linear map $\phi:\cW\to \cW'$. Let us first verify that we have a well-defined map. Assume
$$\sum_jf_j\cc A^{w_j}[a_{1j},\dots, a_{|w_j|j}]\in \cU_\infty^L;$$
then
$$\left(\sum_jf_j\cc A^{w_j}[a_{1j},\dots, a_{|w_j|j}]\right)A^{w_0}\bb=0$$
for all $w_0\in \gp$. Since our representations determine the same series, we have $\cc A^w\bb=\cc'A'^w\bb'$ for every $w\in \gp$, so
$$\left(\sum_jf_j\cc' A'^{w_j}[a_{1j},\dots, a_{|w_j|j}]\right)A'^{w_0}\bb'=0$$
and thus
$$\sum_jf_j\cc' A'^{w_j}[a_{1j},\dots, a_{|w_j|j}]\in \cU_\infty'^L.$$
Therefore the rule \eqref{e:isomap} indeed defines a map; by symmetry, we can reverse these implications, so $\phi$ is injective; it is also linear and surjective by definition. Hence the left modules $\cW$ and $\cW'$ are isomorphic, so the intervals $[m^2n-\delta,m^2n]$ and $[m^2n'-\delta',m^2n']$ overlap by Lemma \ref{l:dimbound}. Therefore we have
\begin{equation}\label{e:ineq}
n\le n' \Rightarrow m^2(n'-n)\le\delta' \quad \text{and} 
\quad n'\le n \Rightarrow m^2(n-n')\le\delta.
\end{equation}

(1) If both representations are reduced, then $\delta$ and $\delta'$ are less than $2m^2$ since torsion $\cA$-module has dimension less than $m^2$; hence \eqref{e:ineq} implies $|n-n'|\in\{0,1\}$ and so (1) holds.

(2) Assume $n\ge n'$; if $\delta<m^2$, then \eqref{e:ineq} implies $n=n'$, so the first representation is minimal.

(3) In this case we have $\cW=\cA^{1\times n}=\cW'$, so the map $\phi$ yields invertible
$P\in\cA^{n\times n}$ such that
\begin{equation}\label{e4}
\cc A^w[a_1,\dots, a_{|w|}]P^{-1}=\cc' A'^w[a_1,\dots, a_{|w|}].
\end{equation}
For $w=1$ we get $\cc P^{-1}=\cc'$. Since
$$\cc A^w[a_1,\dots, a_{|w|}]P^{-1}(PA^x(b)P^{-1}-A'^x(b))=0$$
for all $a_i,b\in\cA$ implies
$$\cc A^w[a_1,\dots, a_{|w|}]P^{-1}(PA^xP^{-1}-A'^x)=0$$
by Lemma \ref{l:difvar} and $\cW=\cA^{1\times n}$, we get $PA^xP^{-1}=A'^x$. Finally, since
$$\cc A^w[a_1,\dots, a_{|w|}]P^{-1}(Pc-c')=0$$
for all $a_i\in\cA$, we get $Pc=c'$.

Uniqueness follows from the fact that any transition matrix satisfies \eqref{e4}.
\end{proof}

\begin{cor}\label{c:cba}
Every minimal representation is reduced. Every totally reduced representation is minimal.
\end{cor}

\begin{proof}
The first part holds by Proposition \ref{p:red}, and the second part is true by Theorem \ref{t:redmin}(2). 
\end{proof}

We end this subsection by observing that totally reduced representations are preserved if the algebra $\cA$ is ampliated.

\begin{lem}\label{l:newsize}
Let $\cA$ be a simple Artinian algebra and $\cA\subseteq \cB$. If $(\cc,A,\bb)$ is a totally reduced representation over $\cA$, then it is also totally reduced over $\cB$.
\end{lem}

\begin{proof}
Since $(\cc,A,\bb)$ is totally reduced, the sets
\begin{equation}\label{e:sets}
\{\cc A^w[a_1,\dots, a_{|w|}]\colon w\in \gp, \ a_i\in\cA\},\qquad \{A^w[a_1,\dots, a_{|w|}]\bb\colon w\in \gp, \ a_i\in\cA\}
\end{equation}
generate $\cA^{1\times n}$ and $\cA^{n\times 1}$, respectively, by Lemma \ref{l:dimbound}. In particular, the $\cA$-linear spans of sets \eqref{e:sets} contain the standard bases of $\F^{1\times n}$ and $\F^{n\times 1}$, respectively. But then the right and left obstruction module of $(\cc,A,\bb)$ as a representation over $\cB$ are trivial.
\end{proof}


\subsection{Hankel module}\label{subs:hankel}

In this subsection let $\cA=M_m(\F)$.

In the classical realization theory, i.e., $m=1$, to each formal series $S$ one can assign the Hankel matrix $H_S$ (\cite[Section 2.1]{BR} or \cite[Section 10]{BGM}). Namely, $H_S$ is the infinite matrix over $\F$ with rows and columns indexed by words over $\ulx$, whose $(u,v)$-entry equals the coefficient of the word $uv$ in $S$ for every $u,v\in\gs$. The important result  is that $S$ is rational if and only if $H_S$ has finite row rank, 
see e.g. \cite[Theorem 2.6.1]{BR}. By the definition of a stable module in 
\cite[Section 1.5]{BR}, we see that the rows of $H_S$ are generators of the smallest stable module containing $S$, so the rank of $H_S$ equals the minimal cardinality of a generating set of the smallest stable module containing $S$.

Now let $m\in\N$ be arbitrary. From our definition of stability it is not a priori evident that the smallest stable left $\cA$-submodule containing $S$ even exists; namely, it is not clear whether the intersection of stable submodules is again stable.

Recall the map $L:\gp\to\End_{\cA-\cA}(\ags)$ defined by \eqref{e:opL}. For $x\in \ulx$ and $a\in\cA$ we now define closely related operators
\begin{equation}\label{e:opLa}L_{x,a}S={}_a(L_xS),\end{equation}
where ${}_aT$ means that we replace the first letter with $a$ in every term of the series $T$.

\begin{lem}\label{l:geostable}
Let $x\in \ulx$, $S\in\ags$ and $\cM$ be a left $\cA$-submodule in $\ags$. If $L_{x,e_{ij}}S\in\cM$ for all $1\le i,j\le m$, then $L_xS\in\cA^x\cM$.
\end{lem}

\begin{proof}
Let
$$S'=\sum_{i,j=1}^m\sum_{k=1}^m e_{ki}xe_{jk}L_{x,e_{ij}}S.$$
By assumption, $S'\in \cA^x\cM$. By the way we defined this series, we have ${}_{e_{ij}}(L_xS-S')=0$ for all $1\le i,j\le m$. Therefore
${}_a[L_xS-S',w]=0$ for all $a\in \cA$ and $w\in \gp$, so $[L_xS-S',w]=0$ by Lemma \ref{l:difvar}.
\end{proof}

\begin{cor}\label{c:minstable}
For every $S\in\ags$ there exists the smallest stable left $\cA$-submodule containing $S$. It is generated by the series
$$L_{x_{i_1},e_{i_1j_1}}\cdots L_{x_{i_\ell},e_{i_\ell j_\ell}}S,$$
where $\ell\in\N\cup\{0\}$ and $1\le i_\ell,j_\ell\le m$.
\end{cor}

\begin{proof}
The proposed submodule is obviously contained in every stable submodule $\cM$ containing $S$. On the other hand, it is stable by Lemma \ref{l:geostable}.
\end{proof}

The module of Corollary \ref{c:minstable} is called the \emph{Hankel module} of the series $S$ and is denoted $\cH_S$, following the analogy with Hankel matrices described in the first paragraph.

\begin{thm}\label{t:hankel}
Let $S\in\ags$ and let $\cH_S$ be its Hankel module. Then $\cH_S$ is finitely generated as a left $\cA$-module if and only if $S$ is recognizable.

If this is the case, then the minimal cardinality of a generator set of $\cH_S$ equals the dimension of a minimal representation of $S$.
\end{thm}

\begin{proof}
($\Rightarrow$) Since $\cH_S$ is stable, this is a special case of 
Proposition \ref{p:stable}; from its proof it can be also deduced that if $\cH_S$ is generated by $n$ elements, then $S$ has a representation of dimension $n$.

($\Leftarrow$) Let $n$ be the dimension of a representation $(\cc,A,\bb)$ of $S$. By the description of $\cA$-modules in Subsection \ref{subs:sa}, the submodule 
$\cW_0\subseteq \cA^{1\times n}$ as in \eqref{e:W0W} from Lemma \ref{l:dimbound} can be generated 
by $n$ elements. Comparing the definition of $\cW_0$ with the generators of $\cH_S$ in \ref{c:minstable} it can be observed that
$$\cH_S=\cW_0 \left(I_n-\sum_x A^x\right)^{-1}\bb,$$
which finishes the proof.
\end{proof}


\section{Realizations of noncommutative rational functions}\label{sec5}

In this section we apply the theory of linear representations from previous sections to nc rational functions with coefficients in $\F$. 

First we introduce realizations about points (Corollary \ref{c:realization}) and prove the existence of totally reduced realizations (Theorem \ref{t:totred}). Similarly to \cite{KVV1}, we describe the domain of a rational function using its totally reduced realization and furthermore show that the dimension of a minimal realization about a point is an invariant of the rational function, i.e., it is independent of the choice of the point of expansion. We call it the Sylvester degree of a rational function.


\subsection{Definition of realization}\label{subs:def}

We are now in position to explain the meaning of \eqref{e:sylvester} from the introduction. If $r\in \re$ is defined at $\ulp\in \cA^g$, then $S=r(x_1+p_1,\dots,x_g+p_g)$ is a generalized rational expression as in Subsection \ref{subs:eval} and therefore also a rational series in $\ags$. We call it \emph{the expansion of $r$ about $\ulp$}. By Theorem \ref{t:equiv}, $S$ has linear representation. Its entries belong to the rational closure of $\{p_1,\dots,p_d\}$ and $\F$ in $\cA$.

\begin{cor}\label{c:realization}
Let $r\in \re$ and $\ulp\in\dom_m r$. If $\cA\subseteq M_m(\F)$ is the unital $\F$-algebra generated by $p_1,\dots,p_g$, then there exist $n\in \N$, 
$\cc\in \cA^{1\times n}$, $\bb\in \cA^{n\times 1}$ and $A^{x_j}\in (\cA^{x_j})^{n\times n}$ such that
\begin{equation}\label{e:prereal}
r(\ulx+\ulp)=\cc\left(I_n-\sum_{j=1}^g A^{x_j}\right)^{-1}\bb.
\end{equation}
\end{cor}

If \eqref{e:prereal} holds, we say that the linear representation $(\cc,A,\bb)$ is a \emph{realization of $r$ of dimension $n$ about a point $\ulp\in M_m(\F)^g$}.

\begin{rem}\label{r:bound}
Suppose we are given $r\in\re$. Then the content of Subsection \ref{subs:arith} yields a realization of $r$ about any point in $\dom r$ and thus an upper bound on the minimal dimension of a realization for $r$, even if we do not know any elements of $\dom r$. We say that this representation is obtained by \emph{standard construction}. Moreover, we can easily calculate this particular upper bound, denoted $\varkappa(r)$, just by counting the occurrences of symbols in the string determining $r$:
$$\varkappa(r)=\# \text{constant\_terms}+2\cdot \# \text{letters}+\# \text{inversions}.$$
\end{rem}

Equality \eqref{e:prereal} implies that for every $s\in\N$ and 
$\ulq\in\dom_{sm} r$ we have
$$r(\ulq)=\iota(\cc)\left(I_{smn}-\sum_{j=1}^g \iota(A^{x_j})(q_j-\iota(p_j))\right)^{-1}\iota(\bb),$$
where $\iota$ is the entry-wise applied embedding $M_m(\F)\to M_{sm}(\F)$ given by $a\mapsto I_s\otimes a$. For this reason we usually write
\begin{equation}\label{e:real}
r=\cc\left(I_n-\sum_{j=1}^g A^{x_j}(z_j-p_j)\right)^{-1}\bb
\end{equation}
instead of \eqref{e:prereal}. Also, since $A^x$ can be written as a finite sum of terms of the form $C x B$ for $C,B\in M_m(\F)^{n\times n}$ (more precisely, $\dim_{\F} M_m(\F)^{n\times n}=(mn)^2$ terms suffice), we can rewrite \eqref{e:real} as
$$r=\cc\left(I_n-\sum_{j=1}^g\sum_k C_{jk}(z_j-p_j)B_{jk}\right)^{-1}\bb$$
with inner sums being finite and $C_{jk},B_{jk}\in M_m(\F)^{n\times n}$,
which is the \emph{Sylvester realization} \eqref{e:sylvester} mentioned in the introduction.

\begin{exa}\label{e:invcomm}
Consider $r=(z_1z_2-z_2z_1)^{-1}$. Let $(p_1,p_2)$ be an arbitrary pair of elements from some simple Artinian ring, whose commutator $p_1p_2-p_2p_1$ is invertible, and set $q=r(p_1,p_2)$. If $x_i=z_i-p_i$, then
$$r(z_1,z_2)=q(1-(p_2x_1-x_1p_2)q-(x_2p_1-p_1x_2)q-(x_2x_1-x_1x_2)q)^{-1}.$$
It is easy to verify that
\begin{equation}\label{e:favourite}
r=\cc\Big(I_3-C_1(z_1-p_1)B_1-C_2(z_2-p_2)B_2\Big)^{-1}\bb,
\end{equation}
where
$$
\cc=\begin{pmatrix}q &0&0\end{pmatrix},\quad
C_1=\begin{pmatrix}
-1 &p_2 &0 \\
0 &0 &0 \\
0 &-1 &0
\end{pmatrix},\quad
B_1=\begin{pmatrix}
p_2q & -1 &0 \\
q &0 &0 \\
0 &0 &0
\end{pmatrix},
$$
$$
C_2=\begin{pmatrix}
1 &0 &-p_1 \\
0 &0 &-1 \\
0 &0 &0
\end{pmatrix},\quad
B_2=\begin{pmatrix}
p_1q &0 &-1 \\
0 &0 &0 \\
q &0 &0
\end{pmatrix},
\quad \bb=\begin{pmatrix}1\\0\\0\end{pmatrix}.
$$
Therefore $r$ has a realization of dimension $3$. On the other hand, Remark \ref{r:bound} guarantees only a realization of dimension $\varkappa(r)=9$.

Moreover, the realization \eqref{e:favourite} is totally reduced 
(and therefore also minimal) since 
$$\{\cc,\cc C_1B_1,\cc C_2B_2\}\qquad \text{and}\qquad \{\bb,C_1B_1\bb,C_2B_2\bb\}$$
are bases for $\cA^{1\times 3}$ and $\cA^{3\times 1}$, respectively.
\end{exa}

Let $r_1,r_2\in\re$ determine the same rational function and let $\ulp\in\dom r_1\cap \dom r_2$. If $S_i$ is the expansion of $r_i$ about $\ulp$ for $i\in\{1,2\}$, then $S_1=S_2$ by Proposition \ref{p:zeroseries}. Therefore we can define a \emph{realization of a nc rational function $\rr\in\rf$ about a point $\ulp\in \dom \rr$} as the realization of $r\in\re$ about $\ulp$, where $r$ is an arbitrary rational expression representing $\rr$ that is defined at $\ulp$.

In the following subsections we concentrate on matrix-valued evaluations of nc rational functions. Therefore set $\cA=M_m(\F)$ for the rest of the section.


\subsection{Existence of totally reduced realizations}

In this subsection we prove the existence of totally reduced realizations of a nc rational function about ``almost every'' point from in domain; in fact, it is shown that we already have the means of producing these realizations.

Recall the generic $m\times m$ matrices $T_k$ and the generic division algebra $\ud\subset M_m(\F(t))$ generated by them from Section \ref{sec2}.

\begin{lem}\label{l:locglob}
Assume $\vv_1,\dots, \vv_k\in M_m(\F(t))^{1\times n}$ are left $M_m(\F(t))$-linearly independent. Then there exists a nonempty Zariski open set $\cO\subseteq \cA^g$ such that for 
$\ulxi\in\cO$, the rows $\vv_1(\ulxi),\dots, \vv_k(\ulxi)\in M_m(\F)^{1\times n}$ are left 
$M_m(\F)$-linearly independent.
\end{lem}

\begin{proof}
The set $\{\vv_1,\dots, \vv_k\}$ can be extended to a basis, so there exists an invertible matrix $V\in M_n(M_m(\F(t)))$ whose rows belong to this basis. Considering this matrix as an element of $M_{nm}(\F(t))$, it is clear that for every tuple $\ulxi$ over $\F$, such that $V$ and $V^{-1}$ are defined at $\ulxi$, the matrix $V(\ulxi)\in M_{nm}(\F)$ is also invertible. Note that the set of such tuples is Zariski open and nonempty since $\F$ is infinite. Finally, considering $V(\ulxi)$ as an element of $M_n(M_m(\F))$ finishes the proof.
\end{proof}

\begin{thm}\label{t:totred}
Assume $r\in\re$ is defined at some point in $\cA^g$. Then there exists a nonempty Zariski open set $\cO\subseteq \dom_m r$ such that for every $\ulp\in \cO$, the formal expansion $S$ of $r$ about $\ulp$ admits a totally reduced representation.

Furthermore, this representation can be obtained by applying the proof of Proposition \ref{p:red} to the standard construction of a representation of $S$.
\end{thm}

\begin{proof}
By assumption, $r$ is also defined at the tuple $\ulT=(T_1,\dots,T_g)$ of $m\times m$ generic matrices over $\F$. Let $(\cc,A,\bb)$ be the standard construction of realization of $S$ about 
$\ulT$; then the standard construction of realization of $S$ about arbitrary $\ulp\in\dom_m r\cap\cA^g$ is just the evaluation of $(\cc,A,\bb)$ at $\ulp$, which we denote $(\tilde{\cc},\tilde{A},\tilde{\bb})$.

Since $\ud$ is a division algebra, all modules over $\ud$ are free, hence the reduction of $(\cc,A,\bb)$ ends after two steps. Let $n'$ and $\cU'^R_\infty$ be the dimension and the right obstruction module of representation after the first step of reduction process, respectively. Set
\begin{align*}
\cV&=\{\vv \in \cA^{n\times 1}\colon \uu\vv =0\ \forall \uu\in \cU_\infty^L\},\\
\cV'&=\{\vv \in \cA^{1\times n'}\colon \vv\uu =0\ \forall \uu\in \cU'^R_\infty\}.
\end{align*}
Let $B^{(1)},B^{(2)},B^{(3)},B^{(4)}$ be some bases of free $\ud$-modules $\cU_\infty^L,\cU'^R_\infty,\cV,\cV'$, respectively. Note that
\begin{equation}\label{e:dimeq}
|B^{(1)}|+|B^{(3)}|=n,\quad |B^{(2)}|+|B^{(4)}|=n'.
\end{equation}
Let $\cO_i\subseteq \dom_m r \cap \cA^g$ be the set of all tuples at which the set $B^{(i)}$ evaluates to a linearly independent set $B^{(i)}(\ulp)$ as in Lemma \ref{l:locglob}. 
Set $\cO=\cO_1\cap \cO_2\cap\cO_3\cap \cO_4$. For each $p\in \cO$, we have
$$B^{(1)}(\ulp)\subset \tilde{\cU}_\infty^L,\ B^{(2)}(\ulp)\subset \tilde{\cU}'^R_\infty, \ 
B^{(3)}(\ulp)\subset \tilde{\cV}, \ B^{(4)}(\ulp)\subset \tilde{\cV}'.$$
By equalities \eqref{e:dimeq} the $\cA$-modules $\tilde{\cU}_\infty^L$ and $\tilde{\cU}'^R_\infty$ are free. Therefore $(\tilde{\cc},\tilde{A},\tilde{\bb})$ reduces to totally reduced representation by Corollary \ref{c:redfree}.
\end{proof}


\subsection{Stable extended domain and Sylvester degree of a nc rational function}

The aim of this subsection is to describe the stable extended domain of a nc rational function using totally reduced realizations and to show that the dimension of its minimal realizations about a point is actually independent of the chosen point. The latter makes it possible to define the Sylvester degree of a nc rational function.

We start by introducing the right-hand version of the map $L$ defined by \eqref{e:opL} and the operators $L_{x,a}$ defined by \eqref{e:opLa}. That is, we define 
$R:\gp\to \End_{\cA-\cA}(\ags)$ by
$$R_vS=\sum_{w\in\gp}[S,wv]$$
and
$$R_{x,a}S=(R_xS)_a$$
for $a\in\cA$, where $T_a$ means that we replace the last letter in every term of series $T$ with $a$. By \eqref{e:L} and its right-hand analog we get
\begin{equation}\label{e:shift1}
\begin{split}
L_{x,a}(ST)=(L_{x,a}S)T+[S,1]L_{x,a}T, \quad &L_{x,a}U^{-1}=-[U,1]^{-1}(L_{x,a}U)U^{-1},\\
R_{x,a}(ST)=(R_{x,a}S)[T,1]+SR_{x,a}T, \quad &R_{x,a}U^{-1}=-U^{-1}(R_{x,a}U)[U,1]^{-1}.
\end{split}
\end{equation}

The linearity of $L_{x,a}$ and $R_{x,a}$ and the equalities \eqref{e:shift1} imply the following: if $S$ is a gr expression (as defined is Subsection \ref{subs:eval}), then $L_{x,a}S$ and $R_{x,a}S$ are also gr expressions and their domains include the domain of $S$. Moreover, if $(\cc,A,\bb)$ is a representation of $S$, then $(\cc A^x(a),A,\bb)$ and $(\cc,A,A^x(a)\bb)$ are representations of 
$L_{x,a}S$ and $R_{x,a}S$, respectively.

Lemma \ref{l:shift2}, Corollary \ref{c:shift3} and Theorem \ref{t:sing} and their proofs closely follow \cite[Theorem 2.19, Corollary 2.20, Theorem 3.1]{KVV1} and \cite[Lemma 3.4, Theorem 3.5]{Vol}.

\begin{lem}\label{l:shift2}
Let $S$ be a gr expression, $\ulp\in\cA^g$ a point in its domain and $\ulq\in\cA^g$ an arbitrary point. Then
\begin{equation}\label{e:kvv}
S\left(\begin{bmatrix}\ulp & 0\\ \ulq & 0\end{bmatrix}\right)=
\begin{bmatrix}S(\ulp) & 0\\ \sum_{i=1}^g(L_{x_i,q_i}S)(\ulp) & S(0)\end{bmatrix}.
\end{equation}
\end{lem}

\begin{proof}
We proceed inductively using \eqref{e:shift1}. The claim obviously holds if $S$ is a constant or a letter. If the claim holds for $S_1$ and $S_2$, it also holds for any linear combination of $S_1$ and $S_2$ by linearity of the operators $L_{x,a}$. Moreover, it holds for their product:
\begin{align*}
(S_1S_2)\left(\begin{bmatrix}\ulp & 0\\ \ulq & 0\end{bmatrix}\right)
&=S_1\left(\begin{bmatrix}\ulp & 0\\ \ulq & 0\end{bmatrix}\right)
S_2\left(\begin{bmatrix}\ulp & 0\\ \ulq & 0\end{bmatrix}\right) \\
&=\begin{bmatrix}S_1(\ulp) & 0\\ \sum_{i=1}^d(L_{x_i,q_i}S_1)(\ulp) & S_1(0)\end{bmatrix}
\begin{bmatrix}S_2(\ulp) & 0\\ \sum_{i=1}^d(L_{x_i,q_i}S_2)(\ulp) & S_2(0)\end{bmatrix} \\
&=\begin{bmatrix}S_1(\ulp)S_2(\ulp) & 0\\
\sum_{i=1}^g\left((L_{x_i,q_i}S_1)(\ulp)S_2(\ulp)+S_1(0)(L_{x_i,q_i}S_2)(\ulp)\right)
 & S_1(0)S_2(0)\end{bmatrix} \\
&=\begin{bmatrix}(S_1S_2)(\ulp) & 0\\ \sum_{i=1}^g(L_{x_i,q_i}(S_1S_2))(\ulp) & (S_1S_2)(0)\end{bmatrix}.
\end{align*}
If the claim holds for an invertible series $S$, then
\begin{align*}
S^{-1}\left(\begin{bmatrix}\ulp & 0\\ \ulq & 0\end{bmatrix}\right)
&=S\left(\begin{bmatrix}\ulp & 0\\ \ulq & 0\end{bmatrix}\right)^{-1} \\
&=\begin{bmatrix}S(\ulp)^{-1} & 0\\ 
-S(0)^{-1}\left(\sum_{i=1}^g(L_{x_i,q_i}S)(\ulp)\right) S(\ulp)^{-1}
 & S(0)^{-1}\end{bmatrix} \\
&=\begin{bmatrix}S^{-1}(\ulp) & 0\\ \sum_{i=1}^g(L_{x_i,q_i}S^{-1})(\ulp) & S(0)^{-1}\end{bmatrix}. \qedhere 
\end{align*}
\end{proof}

\begin{cor}\label{c:shift3}
The stable extended domain of a rational series $S$ is included in the stable extended domain of 
$L_{x,a}S$ and $R_{x,a}S$ for every $x\in \ulx$ and $a\in\cA$.
\end{cor}

\begin{proof}
By the definition of the stable extended domain it suffices to prove $\edomst S\subseteq \edom L_{x,a}S$. Assume $\ulp\in\edomstn{ms} S$. Since $\edomst S$ is closed under direct sums, we have $\ulp\oplus 0\in\edomstn{2ms}S$. If $\ulq\in\cA^g$ be arbitrary, then the properties of arithmetic operations on block lower triangular matrices imply
$$\begin{bmatrix}\ulp & 0\\ \ulq & 0\end{bmatrix}\in\edom_{2ms}S.$$
Evaluation of \eqref{e:kvv} on a tuple of generic $m\times m$ matrices $\ulT=(T_1,\dots,T_g)$ yields
\begin{equation}\label{e:kvv1}
S\left(\begin{bmatrix}\ulT & 0\\ \ulq & 0\end{bmatrix}\right)=
\begin{bmatrix}S(\ulT) & 0\\ \sum_{i=1}^g(L_{x_i,q_i}S)(\ulT) & S(0)\end{bmatrix}.
\end{equation}
Both sides of \eqref{e:kvv1} are $2ms\times 2ms$ matrices of rational functions in $g(ms)^2$ commuting variables. Since the left-hand side is defined in $p$, the same holds for the right-hand side. Let $j\in\{1,\dots,g\}$ and $a\in\cA$ be arbitrary; then considering $\ulq=(0,\dots,a,\dots,0)$ with $a$ on position $j$ implies that $\ulp\in\edom_{ms} L_{x_i,a}S$.

Statement for $R_{x,a}S$ is proved analogously using the right-hand version of Lemma \ref{l:shift2}.
\end{proof}

\begin{thm}\label{t:sing}
Let $S$ be a rational series and assume it admits a totally reduced representation $(\cc, A, \bb)$ of dimension $n$. Then
$$\edomstn{m} S=\left\{\ulp\in \cA^g\colon
\det\left(I_{mn}-\sum_{j=1}^g A^{x_j}(p_j)\right)\neq0\right\}.$$
\end{thm}

\begin{proof}
The inclusion $\supseteq$ is obvious. Conversely, assume $\ulp\in\edomstn{m} S$. By Corollary \ref{c:shift3}, it also belongs to the stable extended domain of 
$L_{x_{i_l},a_l}\cdots L_{x_{i_1},a_1} S$ for all combinations of letters and elements of $\cA$. Note that this series can be written as
$$\cc A^{x_{i_1}}(a_1)\cdots A^{x_{i_l}}(a_l)\left(I_n-\sum_{j=1}^g A^{x_j}\right)^{-1}\bb.$$
Since $\cU_\infty^R=0$, the rows $\cc A^{x_{i_1}}(a_1)\cdots A^{x_{i_l}}(a_l)$ generate left $\cA$-module $\cA^{1\times n}$
by Lemma \ref{l:dimbound}. Therefore $\ulp$ belongs to the stable extended domains of the entries in the column
$$\left(I_n-\sum_{j=1}^g A^{x_j}\right)^{-1}\bb.$$
Now we apply $R_{x,a}$ simultaneously on the entries of this column; by the same reasoning as above, $\ulp$ lies in the stable extended domains of the entries in the matrix
$$\left(I_n-\sum_{j=1}^g A^{x_j}\right)^{-1}.$$
Let $Q=\left(I_n-\sum_{j=1}^g A^{x_j}(T_j)\right)^{-1}$; since $\ulp$ is a regular point for every entry in $Q$, it is also a regular point for $\det Q$. Since
$$\det Q\det\left(I_{mn}-\sum_{j=1}^g A^{x_j}(T_j)\right)=1$$
and $\ulp$ is a regular point for both sides of this equation, we can apply it to $p$ and obtain
\[
\det\left(I_{mn}-\sum_{j=1}^g A^{x_j}(p_j)\right)\neq0.
\qedhere \]
\end{proof}

\begin{cor}\label{c:dom}
Let $\rr\in\rf$, $\ulp\in\dom_m \rr$ and $(\cc,A,\bb)$ its totally reduced realization about $\ulp$. Then
$$\edomstn{ms} \rr=\left\{\ulq\in M_{ms}(\F)^g\colon
\det\left(I_{msn}-\sum_{j=1}^g \iota(A^{x_j})(q_j-\iota(p_j))\right)\neq0\right\}$$
for all $s\in\N$, where $\iota:M_m(\F)\to M_{ms}(\F)$ is the embedding $a\mapsto I_s\otimes a$.
\end{cor}

\begin{proof}
Let $r\in \re$ a rational expression corresponding to $\rr$ and defined in $\ulp$. First observe that if $S$ is the expansion of $r$ about $\ulp$, then $\edomstn{m} r=\ulp+\edomstn{m}S$. The case $s=1$ is therefore a direct consequence of Theorem \ref{t:sing}. For arbitrary $s\in\N$, we consider $\ulp$ and $(\cc,A,\bb)$ as a point and a realization over $M_{ms}(\F)$ using inclusion $\iota$ and then apply Lemma \ref{l:newsize} and Theorem \ref{t:sing}.
\end{proof}

\begin{thm}\label{t:indepmin}
Let $\rr\in\rf$. Then the dimension of its minimal realization does not depend on the choice of the expansion point $\ulp\in\dom \rr$.
\end{thm}

\begin{proof}
It is enough to prove the statement for $r\in\re$ because the domains of any two elements of $\re$ have nonempty intersection.

We first consider points that admit totally reduced realizations of $r$. For $i\in\{1,2\}$, let $\ulp_i\in\dom_{m_i} r$ and $(\cc_i,A_i,\bb_i)$ a totally reduced realization of $r$ about $\ulp_i$ of dimension $n_i$. We need to prove $n_1=n_2$. Using Lemma \ref{l:newsize} and working over $M_{m_1m_2}(\F)$, we can assume $m_1=m_2=:m$. By Theorem \ref{t:sing}, the matrix
$$Q=I_{mn_1}-\sum_{j=1}^g A_1^{x_j}(p_{2j}-p_{1j})$$
is invertible. Then
\begin{align*}
\cc_1\left(I_{mn_1}-\sum_{j=1}^g A_1^{x_j}(z_j-p_{1j})\right)^{-1}\bb_1
&=\cc_1\left(I_{mn_1}-\sum_{j=1}^g A_1^{x_j}(z_j-p_{2j}+p_{2j}-p_{1j})\right)^{-1}\bb_1 \\
&=\cc_1\left(Q-\sum_{j=1}^g A_1^{x_j}(z_j-p_{2j})\right)^{-1}\bb_1 \\
&=\cc_1\left(I_{mn_1}-\sum_{j=1}^g (Q^{-1}A_1^{x_j})(z_j-p_{2j})\right)^{-1}(Q^{-1}\bb_1)
\end{align*}
and the latter is a realization about $\ulp_2$. Therefore $n_1\ge n_2$. By symmetry we get $n_2\ge n_1$.

Now let $\ulp\in \dom_m r$ be arbitrary with corresponding minimal realization $(\cc,A,\bb)$ of dimension $n$. Let $\cO$ be the nonempty Zariski open set from Theorem \ref{t:totred} and $n'$ the dimension of totally reduced realizations of $r$ about points from $\cO$. Choose arbitrary $\ulp_0\in\cO$; using the same reasoning as before we can use a totally reduced realization of $r$ about $\ulp_0$ to produce a realization of $r$ about $\ulp$; therefore $n\le n'$. On the other hand, the set
$$\left\{\ulp'\in M_m(\F)^g\colon
\det\left(I_{mn}-\sum_{j=1}^g A^{x_j}(p'_j-p_j)\right)\neq0\right\}$$
is Zariski open and nonempty, so its intersection with $\cO$ is nontrivial. Therefore there exists a point $\ulp'$ admitting totally reduced realization of $r$ such that the matrix
$$Q=I_{mn}-\sum_{j=1}^g A^{x_j}(p'_j-p_j)$$
is invertible. Hence we can get a realization about $p'$ of dimension $n$ by the same reasoning as at the beginning of the proof, which implies $n\ge n'$ because every totally reduced realization is minimal by Corollary \ref{c:cba}. 
\end{proof}

Theorem \ref{t:indepmin} justifies the following definition.

\begin{defn}
The \emph{Sylvester degree} of $\rr\in\rf$ is the dimension of some minimal realization of $\rr$ about an arbitrary point in its domain.
\end{defn}

A systems theory equivalent is called McMillan degree \cite[Section 8.3]{AM}. There are also other integer invariants that describe the complexity of a nc rational function, such as depth \cite[Section 7.7]{Co1} or inversion height \cite{Re}. However, we would like to emphasize the fact that the Sylvester degree can be effectively calculated since we obtain a reduced realization through the algorithm of Proposition \ref{p:red}.

\begin{rem}
Let $D$ be a division ring that is infinite dimensional over its center. In \cite{CR} the authors define the \emph{rank} of a nc rational function $\rr$ as the dimension of its minimal realization about a point in $D^g$. We give a sketch of the proof that the rank and the Sylvester degree coincide (in particular, the rank does not depend on $D$). By Theorem \ref{t:totred}, the Sylvester degree of $\rr$ is equal to the dimension on the minimal realization of $\rr$ about the tuple of generic $m\times m$ matrices; since $\ud$ is the relatively free skew field of degree $m$ \cite[Section 7.5]{Co}, the Sylvester degree of $\rr$ coincides with the minimal dimension of a realization of $\rr$ over any division ring of degree $m$. Recall that totally reduced realizations are preserved under inclusion by Lemma \ref{l:newsize} and that minimal realizations over division rings are totally reduced; since there exists a division ring containing both $D$ and some division ring of degree $m$ \cite[Lemma 15]{Ami}, the Sylvester degree and the rank coincide by \cite[Theorem 4.3]{CR}.
\end{rem}



\section{Applications}\label{sec6}

We conclude the paper with some applications of our results. In Corollary \ref{c:RIbound} we provide size bounds for rational identity test that do not rely on the information about realizations. Next we prove the local-global principle for the linear dependence of nc rational functions and obtain size bounds for testing linear dependence; see Theorem \ref{t:lg}. Finally, we consider symmetric rational expressions and symmetric realizations, following \cite{HMV}.


\subsection{Rational identity testing}\label{subs:RI}

An important aspect of the noncommutative rational expressions is recognizing the rational identities \cite{Ami, Be1,Be2}. As already mentioned in the proof of Proposition \ref{p:GPI}, it is well known that there are no nonzero polynomial identities on $M_s(\F)$ of polynomial degree less than $2s$. In this section, we derive a result of the same type for rational expressions.

\begin{thm}\label{t:RIbound}
Assume $r\in \re$ admits a realization of dimension $n$ about a point in $M_m(\F)^g$. If $r$ vanishes on $\dom_N r$ for $N=m\lceil\tfrac{mn}{2}\rceil$, then $r$ is a rational identity.
\end{thm}

\begin{proof}
We refine the proof of Proposition \ref{p:zeroseries}. Let $(\cc,A,\bb)$ be a realization of $r$ about $\ulp\in M_m(\F)^g$ of dimension $n$ and let $\ulT=(T_1,\dots, T_g)$ be the tuple of generic $N\times N$ matrices. Then $r(\ulT)$ is a matrix of commutative rational functions and the matrix of their expansions about $\ulp$ is
$$M(\ulT)=\sum_w\cc\left(A^w(\ulT-\ulp)\right)\bb.$$
The formal differentiation of commutative power series yields
\begin{equation}\label{e:der1}
\frac{\der}{\der t^h} M\big(p+t(\ulT-\ulp)\big)\Big|_{t=0}
=h!\sum_{|w|=h}\cc \left(A^w(\ulT-\ulp)\right)\bb.
\end{equation}
As $r$ vanishes on $\dom_N r$, we have $r(\ulT)=0$ and hence $M(\ulT)=0$; therefore the left-hand side of \eqref{e:der1} equals 0 for every $h$ and so the same holds for the right-hand side of \eqref{e:der1}. Since $\sum_{|w|=h}\cc (A^w(\ulz-\ulp))\bb$ is a generalized polynomial of degree $h$, we have $\cc A^w\bb=0$ for all $|w|<mn$ by Proposition \ref{p:GPI}. Therefore $\cc$ belongs to the left obstruction module of $(\cc,A,\bb)$ by Lemma \ref{l:termination}, so $(\cc,A,\bb)$ represents the zero series and hence $r$ is a rational identity.
\end{proof}

As an example of an application of Theorem \ref{t:RIbound} we refer to the size bounds for testing whether a polynomial belongs to a rationally resolvable ideal \cite[Section 3]{KVV4}.

The next corollary gives us explicit size bounds for testing whether a given rational expression is a rational identity, which are independent of the realization terminology. 

\begin{cor}\label{c:RIbound}
Let $r\in \re$ and assume $\dom_m r\neq\emptyset$. Then $r$ is a rational identity if and only if $r$ vanishes on $\dom_N r$, where
$$N=m\left\lceil\tfrac{m}{2}\varkappa(r)\right\rceil,$$
and $\varkappa$ is as in Remark \ref{r:bound}.
\end{cor}

\begin{proof}
Direct consequence of Theorem \ref{t:RIbound} and Remark \ref{r:bound}.
\end{proof}

If we have no information about the domain of a nc rational expression $r$ and we want to know if it represents a nonzero nc rational function, then we first have to check whether $s$ represents a nonzero nc rational function for every $s\in\re$ such that $s^{-1}$ is a sub-expression in $r$. In this case it might be preferable to apply rational identity testing results of \cite{HW,DM1}.

The following refinement of Theorem \ref{t:RIbound} supports the intuition that a nonzero rational function, which vanishes on all matrices of a given size, has to be sufficiently complicated.

\begin{cor}\label{c:RIbound1}
Let $\rr\in \rf\setminus \{0\}$ and assume $\dom_m \rr\neq\emptyset$. If $\rr$ vanishes on $\dom_{km}\rr$, then the Sylvester degree of $\rr$ is at least $2k+1$.
\end{cor}

\begin{proof}
Let $r$ be a representative of $\rr$ with $\dom_m r\neq\emptyset$ and let $(\cc,A,\bb)$ be its totally reduced realization about a tuple of generic matrices $\ulT\in\ud^g$. By the proof of Theorem \ref{t:totred}, the size of $(\cc,A,\bb)$ equals the Sylvester degree of $\rr$. Let $\ulT'$ be a tuple of $km\times km$ generic matrices whose entries are independent of variables in $\ulT$. As in the proof of Theorem \ref{t:RIbound} we conclude that
$$\sum_{|w|=h}\cc\big(A^w(\ulT'-\ulT)\big)\bb=0$$
for every $h$. Note that $\cc,A_j,\bb$ are matrices of commutative rational functions; let $\cO\subseteq M_m(\F)^g$ be the intersection of their domains. For every $p\in\cO$ and $h\in\N$, the generalized polynomial $\sum_{|w|=h}\cc(p) (A(p)^w(\ulz-p))\bb(p)$ vanishes on $km \times km$ matrices. By Proposition \ref{p:GPI} we thus have $\cc(p)A(p)^w\bb(p)=0$ for all $|w|\le 2k-1$. Consequently $\cc A^w\bb=0$ for all $|w|\le 2k-1$ because $\cO$ is Zariski dense in $M_m(\F)^g$. If the size of $(\cc,A,\bb)$ were at most $2k$, then $\cc$ would belong to the left obstruction module of $(\cc,A,\bb)$ by Lemma \ref{l:termination} (note that $(\cc,A,\bb)$ is a realization over a division algebra), which would contradict $\rr\neq0$. Hence the Sylvester degree of $\rr$ is at least $2k+1$.
\end{proof}


\subsection{Local and global linear dependence}

As another application of Theorem \ref{t:RIbound} we give an algebraic proof of local-global principle for the linear dependence of nc rational functions: if all matrix-valued evaluations of some nc rational functions are linearly dependent over $\F$, then these nc rational functions are linearly dependent over $\F$ as elements in $\rf$. This result first appeared in \cite{CHSY}; later it was applied to questions arising in free real algebraic geometry \cite{HKM,OHMP,HMV}.

The original proof, which is nicely elaborated in \cite[Section 5.2]{HKM}, is analytic and it does not provide any bounds on the size of matrices for which the linear dependence has to be checked. For nc polynomials, the algebraic version with bounds is given in \cite{BK}. Using our results, the same technique can be applied to nc rational functions.

The polynomial
$$c_\ell(z_1,\dots,z_\ell;z_1',\dots,z_{\ell-1}')=\sum_{\pi\in S_\ell} \sign(\pi)\,
z_{\pi(1)}z_1'z_{\pi(2)}z_2'\cdots z_{\pi(\ell-1)}z_{\ell-1}'z_{\pi(\ell)}$$
in $2\ell-1$ freely noncommutative variables is known as the \emph{$\ell$-th Capelli polynomial}. It can also be defined recursively by $c_1(z_1)=z_1$ and
\begin{equation}\label{e:cap}
c_\ell(z_1,\dots,z_\ell;z_1',\dots,z_{\ell-1}')=\sum_{j=1}^\ell (-1)^{j-1}
z_jz_1'
c_{\ell-1}(z_1,\dots,z_{i-1},z_{i+1},\dots,z_\ell;z_2',\dots,z_{\ell-1}').
\end{equation}

The Capelli polynomial plays an important role in the theory of polynomial identities; we refer to \cite[Theorem 7.6.16]{Row} for the following result.

\begin{thm}\label{t:cap}
Let $\cA$ be a unital simple ring. Then $a_1\dots,a_\ell\in\cA$ are linearly dependent over the center of $\cA$ if and only if $c_\ell(a_1,\dots,a_\ell;b_1,\dots,b_{\ell-1})=0$ for arbitrary $b_1,\dots,b_{\ell-1}\in\cA$.
\end{thm}

We can now give a new proof of the local-global principle, together with concrete size bounds, which makes the statement even more ``local''.

\begin{thm}\label{t:lg}
Let $\rr_1,\dots,\rr_\ell\in \rf$ and assume that the Sylvester degree of $\rr_j$ is at most $d$ and $\dom_m \rr_j\neq \emptyset$, for all $1\le j\le \ell$. If $\rr_1(\ula),\dots, \rr_\ell(\ula)$ are $\F$-linearly dependent for all $\ula\in \bigcap_j\dom_N\rr_j$, where
$$N=\frac{1}{2}m^2\ell !(3 d + 4),$$
then $\rr_1,\dots,\rr_\ell$ are $\F$-linearly dependent.
\end{thm}

\begin{proof}
By the assumption on the Sylvester degrees, Theorem \ref{t:arithmetic}(1,2) and relations \eqref{e:cap}, the nc rational function
$$\rr=c_\ell(\rr_1,\dots,\rr_\ell;z_1',\dots,z_{\ell-1}')$$
in $g+\ell-1$ variables is of Sylvester degree at most
\begin{align*}
&\ell(d+2)+\ell(\ell-1)(d+2)+\cdots+ \ell(\ell-1)\cdots 2(d+2)+\ell ! d \\
=&\ell !(d+2) \sum_{i=1}^{\ell-1}\frac{1}{i !}+\ell ! d <2\ell !(d+2)+\ell ! d=\ell !(3d+4).
\end{align*}
Since $\dom_m \rr_j$ is a nonempty Zariski open set for every $1\le j\le\ell$, we also have $\dom_m\rr\neq \emptyset$. By Theorem \ref{t:cap} and the assumption about linear dependence we have
$$\rr(\ula;\ulb)=0 \quad\forall (\ula,\ulb)\in \bigcap_j\dom_N\rr_j\times M_N(\F)^{\ell-1},$$
so $\rr=0$ by Theorem \ref{t:RIbound} since $\rr$ vanishes on a Zariski dense set 
in $M_N(\F)^{g+\ell-1}$. Therefore
$$c_\ell(\rr_1,\dots,\rr_\ell;\rs_1,\dots,\rs_{\ell-1})=0\quad 
\forall \rs_1,\dots,\rs_{\ell-1}\in\rf$$
holds, so $\rr_1,\dots,\rr_\ell$ are $\F$-linearly dependent by Theorem \ref{t:cap} applied to $\cA=\rf$.
\end{proof}

\begin{rem}
Analogously to \cite[Theorem 3.7]{BK} and its proof, one can also derive explicit size bounds for testing directional linear dependence.
\end{rem}


\newcommand{\T}{\intercal}
\newcommand{\ret}{\cR_{\F}(\ulz\cup \ulz^\T)}

\subsection{Symmetric realizations}\label{subs:sym}

In this subsection we deal with the symmetric version of a realization. We restrict ourselves to $\F\in\{\R,\C\}$. Let $a^*$ denote the conjugate transpose of $a\in M_m(\F)=\cA$. We extend this involution on $\cA$ to an involution on $\ags$ by setting $x^*=x$ for $x\in \ulx$. We say that a generalized series $S$ is \emph{symmetric} if $S^*=S$. Furthermore, if $A=(a_{ij})_{ij}$ is a matrix over $\ags$, then $A^*:=(a_{ji}^*)_{ij}$.

We say that $\ulp\in\cA^g$ is a \emph{Hermitian tuple} if all its entries are Hermitian matrices. A rational function $\rr\in\rf$ is \emph{symmetric} if $\rr(\ulp)=\rr(\ulp)^*$ for all Hermitian tuples $\ulp\in\dom \rr$.

\begin{rem}\label{r:exists}
For every $r\in\re$ there exists a Hermitian tuple $\ulp\in\dom r$. Also, $r$ represents a symmetric rational function if and only if its expansion about $\ulp$ is a symmetric generalized series.
\end{rem}

\begin{proof}[Sketch of proof]
Let $\ulz^\T=\{z_1^\T,\dots, z_g^\T\}$ be an alphabet. A rational expression $s\in \ret$ is called a \emph{$*$-rational identity} if it vanishes on
$$\dom s\cap \bigcup_{n\in\N}\{(\ula,\ula^*)\colon \ula\in M_n(\F)^g\}.$$
For more information about $*$-rational identities we refer to \cite{DM,Ro}. We claim the following:
\begin{enumerate}
\item if $r_1\in \ret$ is $*$-rational identity, then $r_1$ is a rational identity;
\item if $r_2\in \re$ vanishes on all Hermitian tuples $\ulp\in\dom r_2$, then $r_2$ is a rational identity.
\end{enumerate}
Claim (2) implies the first part of Remark \ref{r:exists} by considering the sub-expressions of $r$ and the second part of Remark \ref{r:exists} by Proposition \ref{p:zeroseries}. Claim (1) implies claim (2) by considering $r_1=r_2(z_1+z_1^\T,\dots, z_g+z_g^\T)$. Therefore we are left with the proof of claim (1).

It is well-known that a rational expression is a rational identity if and only if it is a rational identity on some infinite-dimensional division ring, in which case it is a rational identity on all infinite-dimensional division rings; see e.g. \cite[Theorem 16]{Ami} or \cite[Corollary 8.2.16]{Row}, where this is proved using ultraproducts. On the other hand, $*$-rational identities and rational identities coincide for every infinite-dimensional division ring with involution by \cite[Theorem 1]{DM}. Therefore the only missing link is that a $*$-rational identity is a $*$-rational identity on some infinite-dimensional division ring with involution; this is not hard to prove using the ultraproduct construction, cf. \cite[Corollary 8.2.16]{Row}.
\end{proof}

Let $r\in\re$ and let $\ulp\in \dom_m r$ be a Hermitian tuple. We say that $r$ has a 
\emph{symmetric realization $(J;\cc,A)$ about $\ulp$} if
$$r=\cc\left(J-\sum_{j=1}^g A^{x_j}(z_j-p_j)\right)^{-1}\cc^*$$
where $\cc\in \cA^{1\times n}$, $A^{x_j}=(A^{x_j})^*\in \left(\cA^{x_j}\right)^{n\times n}$ and $J\in \cA^{n\times n}$ is a \emph{signature matrix} ($J^*=J$ and $J^2=I$). In this case $r$ obviously represents a symmetric rational function.

Every symmetric realization $(J;\cc,A)$ yields a realization $(\cc,JA,J\cc^*)$. Therefore the dimension of a symmetric realization of a rational function is at least the Sylvester degree of the rational function.

\begin{thm}\label{t:sym}
Let $\rr\in\rf$ be symmetric and let $\ulp\in \dom_m \rr$ be a Hermitian tuple such that $\rr$ admits a totally reduced realization $(\cc,A,\bb)$ about $\ulp$. Then:
\begin{enumerate}
\item $\rr$ admits a symmetric realization $(J;\cc',A')$ about $\ulp$;
\item its dimension equals the Sylvester degree of $\rr$, its domain equals the stable extended domain 
of $\rr$, and the signature of $J$ does not depend on the choice of a totally reduced representation;
\item every point in the connected component of the domain of $\rr$ containing $\ulp$ has a symmetric realization of minimal dimension with signature matrix $J$.
\end{enumerate}
\end{thm}

\begin{proof}
(1) We follow the proof of \cite[Lemma 4.1(3)]{HMV}. Since $\rr$ is a symmetric expression and $p$ is a Hermitian tuple, $(\bb^*,A^*,\cc^*)$ is also a totally reduced realization of $\rr$ about $\ulp$. By Theorem \ref{t:redmin}, they are similar via a unique transition 
matrix $S$:
$$S\bb=\cc^*,\quad SA^xS^{-1}=(A^x)^*$$
for $x\in \ulx$. Transposing these equations we see that the realizations $(\cc,A,\bb)$ and  $(\bb^*,A^*,\cc^*)$ are also similar via the transition matrix $S^*$. By uniqueness we have $S=S^*$. We decompose $S=RJR^*$, where $J$ is a signature matrix and $R$ is invertible, and set
$$\cc'=\cc (R^*)^{-1},\quad A'^x=JR^*A^x(R^*)^{-1}$$
for $x\in \ulx$. A short calculation shows that $(J;\cc',A')$ is a symmetric realization of $\rr$ about $p$.

(2) The statement about the the stable extended domain follows from Corollary \ref{c:dom}. The signature of $J$ equals the signature of $S$ from the proof of (1), which is unique with respect to $(\cc,A,\bb)$. Every other totally reduced realization is of the form $(\cc U^{-1},UAU^{-1},U\bb)$ by Theorem \ref{t:redmin} and it is easy to verify that the corresponding Hermitian matrix is $(U^{-1})^*SU^{-1}$, which has the same signature as $S$.

(3) Let $\gamma:[0,1]\to \dom \rr$ be an arbitrary path in the space of Hermitian tuples starting at $\ulp$. Let
$$M(t)=I-\sum_{j=1}^gA'_j(\gamma(t)_j-p_j);$$
we claim that the signature of $M(1)$ equals the signature of $J$. If this were not the case, then the continuity of eigenvalues would imply the existence of $t\in(0,1)$ such that $M(t)$ is singular. However, this contradicts the part of (2) about the domain. Therefore $M(1)=R_1JR_1^*$ and it can be verified that $(J;R_1^{-1}A'(R_1^{-1})^*,\cc'(R_1^*)^{-1})$ is a symmetric realization of $\rr$ about $\gamma(1)$.
\end{proof}

\begin{cor}\label{c:sympos}
If a symmetric rational function $\rr$ admits a totally reduced realization about a Hermitian tuple $\ulp$ and the signature matrix of the corresponding symmetric realization is the identity matrix, then $\rr$ is positive on the connected component of its domain containing $\ulp$.
\end{cor}

\begin{exa}
The rational function $(1+z^2)^{-1}$ is positive on the whole domain and its symmetric realization about 0 is
$$\left(
\begin{pmatrix}-1&0\\0&1\end{pmatrix};
\begin{pmatrix}0&-z\\-z&0\end{pmatrix},
\begin{pmatrix}0\\1\end{pmatrix}
\right),$$
so there is no obvious converse of Corollary \ref{c:sympos}.
\end{exa}



\end{document}